\documentclass[12pt]{article}

\usepackage[utf8]{inputenc}
\usepackage[margin=2.5cm]{geometry}%
\usepackage{xcolor}
\usepackage{graphicx}
\usepackage{subcaption}
\usepackage[authoryear]{natbib}

\usepackage{amsmath}%
\usepackage{amsthm}%
\usepackage{amssymb}%
\usepackage{bbm}%
 \usepackage{url}

\graphicspath{ {./img/} }

\title{Detecting relevant changes in  the  spatiotemporal mean function}
\author{
  { Holger Dette, Pascal Quanz} \\
{ Ruhr-Universit\"at Bochum} \\
{ Fakult\"at f\"ur Mathematik} \\
{ 44780 Bochum, Germany}
}

\date{}

\def\br#1{\left({#1}\right)}                      
\def\sqbr#1{\left[{#1}\right]}                    
\def\cbr#1{\left\{{#1}\right\}}                   
\def\gbr#1{\lfloor #1 \rfloor}                    

\def\pad{\phantom{a}}                             
\def\ppad{\phantom{aa}}                           
\def\pppad{\phantom{aaa}}                         
\def\eqo#1{\stackrel{\text{#1}}{=}}               
\def\norm#1{\left\lVert#1\right\rVert}            
\def\innpr#1{\langle #1 \rangle}                  

\def\convd{\xrightarrow{\phantom{a} \mathpcal{D} \phantom{a}}}

\def\convproc{\rightsquigarrow}

\def\P#1{\mathrm{P} \br{#1}}
\def\E#1{\e \sqbr{#1}}
\def\p{\mathrm{P}}
\def\e{\mathrm{E}}

\def\cov{\mathrm{Cov}}

\DeclareMathAlphabet{\mathpcal}{OMS}{pzc}{m}{n}

\DeclareMathOperator*{\argmax}{arg\,max}

\theoremstyle{definition}
\newtheorem{dfn}{Definition} 
\newtheorem{lem}[dfn]{Lemma}
\newtheorem{rem}[dfn]{Remark}
\newtheorem{prop}[dfn]{Proposition}
\newtheorem*{rem*}{Remark}

\newtheorem{thm}[dfn]{Theorem}
\newtheorem*{thm*}{Theorem}

\newtheorem{ass}{Assumption}

\numberwithin{dfn}{section}
\numberwithin{ass}{section}

\begin{document}

    \maketitle

\begin{abstract}
For  a spatiotemporal process 
$\{X_j(s,t)  |  ~s \in S~,~t \in T \}_{j =1, \ldots , n} $, where $S$ denotes the set of spatial locations and $T$ the time domain, we  consider the problem of testing  for a  change in the sequence of mean functions. In contrast to most of the literature  we are not interested in arbitrarily small changes, but  only in  changes with a  norm exceeding a given threshold. 
Asymptotically distribution free 
 tests are proposed, which do not require
the estimation of the long-run spatiotemporal covariance structure. In particular we consider 
a fully functional approach and a test based 
on the cumulative sum paradigm, investigate the large sample properties of the corresponding test statistics and study their finite sample properties by means of simulation study.
\end{abstract} 
  
  \medskip
  \noindent
  Keywords: Spatiotemporal process, functional data analysis, change point analysis, self-normalization, relevant hypotheses

\noindent AMS Subject classification: 
62M10, 62R10

    \section{Introduction}     \label{sec1} 
      \def\theequation{1.\arabic{equation}}
\setcounter{equation}{0}

In many applications such as in  the analysis of weather- or pollution-related data, measurements are obtained at different spatial locations over a certain time period at a high temporal frequency. 
Often there exists a natural segmentation of the time series such that
it is reasonable to model at each 
spatial component, say $s$,  and on each segment, say $j$, the resulting data 
 as a function, say $t \to X_j(s,t)$ of the time  (on the corresponding segment). 
 Typical examples are  measurements at different geographical locations. For example, within the United States Climate Reference Network (USCRN) high resolution infrared surface temperature measurements  at $126$  stations in the US are publicly available on the website of the NOAA U.S. government agency. Here at  each location $s$, and each day $j$ 
one observes the  daily temperature curve $t \to X_j(s,t)$ 
 \citep{Diam13}. Other examples include yearly curves at different locations 
 over different years such as the daily mean temperature records from $1916$ to $2018$ in $40$ representative Canadian cities, which are publicly available from the government of Canada website.
 In these applications   data is typically 
 modelled in the form
 \begin{equation}
     \label{1.1}
     X_j(s,t) ~,~~ s \in S~,~~t \in T~,~~j=1, \ldots , n, 
 \end{equation}
 where $S$ is a finite set, $T$ is a dense set (we will later consider an interval). 
 
  A typical question in this context is, if the mean, say  $\{\mu_j(s,t) |  ~s \in S~,~t \in T \}_{j =1, \ldots , n} $, 
 of a spatiotemporal process 
 $\{X_j(s,t) |  ~s \in S~,~t \in T \}_{j =1, \ldots , n} $ has changed  over a specific time period.
 For a fixed location this corresponds to  the meanwhile classical
 change point problem in functional data analysis
 \citep[see, for example,][among many others]{berkes2009,zhang2011,astonkirch2012,HorvathKokoskza2012,Aue2015DatingSB,dette2018}. On the other hand, in the spatiotemporal context as considered in model \eqref{1.1} the change point problem is not so well studied.
 Recently, \cite{gromenko} proposed a test for the hypothesis of the existence of a change point in  the mean function, say $\mu_j (s,t)  = \mathbb{E} [X_j(s,t)]$,
 in a sequence of independent observations.
 They formulated the null hypothesis and alternative in the form
 \begin{equation*}
     H_0: \mu_1 = \mu_2 = \ldots = \mu_n
 \end{equation*}
 and
   \begin{equation*}
 H_1:  \mu:=
 \mu_1= \ldots = \mu_{\lfloor n \vartheta_0 \rfloor}~\not = ~  
\mu_{\lfloor n \vartheta_0 \rfloor +1} = \ldots = \mu_n   =: \mu + \delta
  \end{equation*}
  for some $\vartheta_0 \in  (0,1)$,
  and combined the  CUSUM principle with classical principal component analysis to construct a test for these hypotheses, which generalizes the approach of
 \cite{berkes2009} to  the spatiotemporal model \eqref{1.1}.  
 We also refer to the recent paper of  \cite{zhao2021composite} who proposed
change point analysis based on  a composite likelihood criterion for a different spatiotemporal model.
 
In contrast to this literature (and also to most of the literature on change point analysis
for functional data) this paper takes a different look at the change point problem. Our work is motivated by the observation  that in many applications one might  not be interested in
arbitrary  ``small''  changes in the mean function  (in fact, one often does not believe that this function is completely constant over the whole time period for all locations). As an alternative we therefore propose to test the hypotheses 
of the existence of a  time point $\lfloor n \vartheta_0 \rfloor$, such  that the difference, say $\delta$,  between the mean functions before and after this point in time
is {\it relevant}. For this purpose we define two measures
of relevance. The first  one corresponds to the fully functional approach as advocated in \cite{Aue2015DatingSB} and is based on 
a norm of
the difference $\delta$. The second one is related to the PCA approach as considered in \cite{berkes2009} and \cite{gromenko} 
and uses the norm of the projection of the difference on  the first 
principle components. The null hypothesis is then stated in the form that the norm is less or equal than a given threshold $\Delta >0$, that is $H_0: \| \delta \|^2 \leq \Delta$ (see Section \ref{sec2} for details). We derive pivotal tests for both testing problems, which 
neither require estimation of the long-run variance  of the process $\{X_j\}_{j=1, \ldots , n}$ nor the the estimation of the covariance structure of the random field $ \{ X_j(s,t) | s \in S,~t \in T\} $. 

In Section \ref{sec2} we introduce the basic terminology and carefully define the two types of hypotheses considered in this paper.
Section \ref{sec3} is devoted to the  fully functional approach, while the problem of testing relevant hypotheses by projections on the functional principal components is investigated in Section \ref{sec4}. Finally,  in 
Section \ref{sec5} we illustrate our approach by means of a small simulation study and by
the analysis of a data example.

\section{Relevant changes in the spatiotemporal mean}
\label{sec2} 
  \def\theequation{2.\arabic{equation}}
\setcounter{equation}{0}

For a finite set  $S$  let $ L^2 (S \times \sqbr{0,1})$ denote the set of all square integrable functions of the form $f:S \times [0,1] \to \mathbb{R} $ with the common inner product
$$
\innpr{f,g} = \sum_{s \in S} \int_T f(s, t) g(s, t) \mathrm{d}t
$$
and corresponding norm  $\| f \| = \innpr{f,f} ^{1/2}$.
  Let 
  $\cbr{X_j}_{j \in \mathbb{N}}$ be a sequence of square integrable random functions on $S \times [0,1]$, where 
  \begin{equation}
      \label{hd1} 
   X_j = \mu_j + \eta_j,~j \in \mathbb{N}~, 
  \end{equation}
 $\cbr{\eta_j}_{j \in \mathbb{N}}$ is a centered error process and $\{ \mu_j \}_{j\in  \mathbb{N}} $ is a sequence of mean functions in $ L^2 (S \times \sqbr{0,1})$.   We 
 assume that  the  mean functions are  of the form 
  \begin{equation*} 
\mu =  \mu_1= \ldots = \mu_{\lfloor n \vartheta_0 \rfloor}~,~~
\mu_{\lfloor n \vartheta_0 \rfloor +1} = \ldots = \mu_n =\mu + \delta,
 \end{equation*}
 where $\mu, \delta$ denote deterministic, but unknown elements in $L^2 (S \times \sqbr{0,1})$ and $\vartheta_0 \in \br{0,1}$ is a potential  (unknown) change point. 
 The case $\delta =0$ corresponds to the situation  of no change point. As explained in the introduction, we  are  not interested in ``small'' deviations before and after a potential change point 
and therefore consider the problem of  monitoring the sequence for a relevant change in the mean function by testing the {\it relevant } hypotheses 
 \begin{equation} \label{hd2}
H_0:  \|  \delta \|^2  \leq \Delta 
~~{\rm versus  } ~~ H_1:  \|  \delta \|^2  >  \Delta.
 \end{equation}
Here $\Delta >0$ is a predefined threshold, which defines the difference before and after 
the time point $\lfloor n \vartheta_0 \rfloor$
 as relevant. Note that 
the case $\Delta=0$ corresponds to the ``classical'' hypotheses 
\citep[see][]{gromenko},
but  this case  is not considered  here.
Our interest in hypotheses of the from \eqref{hd2} with $\delta >0$ 
stems from the fact that in applications it is often questionable  to look 
for arbitrary  small deviations. Instead it is more
reasonable to focus on (scientifically) relevant deviations, which are here defined by  the   threshold $\Delta$ in \eqref{hd2}. The choice of this threshold  depends   sensitively
on the specific application \citep[see Remark \ref{remarkh0} for some discussion and][for an example in the context of portfolio analysis based on multivariate data]{Dette2014b}.
We also note that for hypotheses of the form \eqref{hd2} the choice of the norm matters as objects
might be identified as close with respect to one norm (such as an $L^2$), while they might be considered 
as different with respect to another norm (such as the sup-norm). Moreover, we also mention that 
the null  hypothesis and alternative in \eqref{hd2} can easily be changed,  i.e. 
 \begin{equation} \label{hd2a}
H_0:  \|  \delta \|^2  > \Delta 
~~{\rm versus  } ~~ H_1:  \|  \delta \|^2 \leq    \Delta.
 \end{equation}
This formulation is attractive because  it allows to decide  for a non-relevant change (such that one can continue working under the assumption of a nearly constant mean function) at a controlled type I error.
For real valued data,  hypotheses of the form \eqref{hd2} and \eqref{hd2a} 
have found considerable attention in the literature \citep[see, for example, the monographs of][]{chowliu1992,wellek2010testing}.
This concept has also been used by \cite{liubrehaywynn2009,gsteiger2011}  and \cite{detmolvolbre2015}  to establish the similarity of different parametric regression curves  which are estimated from real valued data.
In the context of functional data analysis, relevant hypotheses have been considered by 
\cite{fogarty2014,dette2018} and \cite{dette} among others. 
A pivotal test for  the hypotheses \eqref{hd2}  (and as a consequence also for the hypotheses \eqref{hd2a}) 
will be developed in Section \ref{sec3}.

Recently, \cite{gromenko} considered a different quantity to measure deviations of the difference $\delta$
from the function $\delta \equiv 0$, which is closely related to functional principal component analysis. More precisely, assume that ${\cal B}:= \{ b_1, b_2 , \ldots  \} $ is a basis in $L^2(S \times [0,1])$, such that the 
linear span of ${\cal B}$ is dense in $L^2(S \times [0,1])$. Then these authors proposed to test, for 
a fixed order $d \in \mathbb{N}$, whether the sum of the squared scores $ \sum_{k=1}^d \langle \delta , b_k\rangle^2 $ vanishes.
In the context of testing relevant hypotheses, we are therefore interested in testing hypotheses
of the form
$$
    H_0: \sum_{k=1}^d \innpr{\delta,  b_k}^2 \leq \Delta~~{\rm versus ~~}  H_1: \sum_{k=1}^d \innpr{\delta, b_k}^2 > \Delta~.
$$
A pivotal test for these hypotheses, where the basis functions are given by the eigenfunctions of a 
convex combination of the covariance kernels before and after the change point, will be developed in Section \ref{sec4}.

\smallskip 

We   conclude this section by presenting several assumptions, which are required 
 to prove the results in the next  and the following sections, 
\begin{ass}\label{ass_2nd} ~~
    \begin{enumerate}
        \item[(A1)] The  process $\cbr{X_j}_{j \in \mathbb{Z}}$ in model \eqref{hd1} satisfies
\begin{align*}
    X_j = \begin{cases}
        \mu + \eta_j^{(1)} & j \leq n \vartheta_0,\\
        \mu + \delta + \eta_j^{(2)}, & j > n \vartheta_0,
    \end{cases}
\end{align*}
where $\{ \eta_j^{(1)} \}_{j \in \mathbb{Z}} $ and  $\{ \eta_j^{(2)} \}_{j \in \mathbb{Z}} $ are stationary processes in $L^2 (S \times \sqbr{0,1})$.

        \item[(A2)] 
        $\{ \eta_j^{(1)} \}_{j \in \mathbb{Z}} $ and  $\{ \eta_j^{(2)} \}_{j \in \mathbb{Z}} $
         form sequences  of Bernoulli shifts, i.e. there exist a measurable space $\mathpcal{S}$,  measurable functions $f_1,f_2 : \mathpcal{S}^\infty \longrightarrow L^2 (S \times \sqbr{0,1})$ and a sequence of i.i.d,  $\mathpcal{S}$-valued
         and jointly (in $(s, t, \omega)$) measurable random functions $\cbr{\varepsilon_j}_{j \in \mathbb{Z}}  = \{ \varepsilon_j (s, t, \omega)
         \}_{j \in \mathbb{Z}} 
         $ 
         such that
         $$\eta_j^{(\ell )}   = f_\ell (\varepsilon_j , \varepsilon_{j-1} , ...)  ~~~~~~~  ~(\ell=1,2) $$
   for all $ j \in \mathbb{Z}$.
        \item[(A3)] There exists a constant   $\psi \in \br{0,1}$ such that $\e \norm{\eta_j^{(\ell )}}^{2 + \psi} < \infty$~ ($\ell =1,2$).
        
        \item[(A4)] The sequences $\{\eta_j^{(1 )}\}_{j \in \mathbb{Z}}$  and $\{\eta_j^{(2 )}\}_{j \in \mathbb{Z}}$ can be approximated by $m$-dependent sequences $\{\eta_{j,m}^{(1 )}\}_{j \in \mathbb{Z}}$  and $\{\eta_{j,m}^{(2 )}\}_{j \in \mathbb{Z}}$, respectively,  in the sense that for some $\kappa > 2 + \psi$ $$\sum_{m=1}^\infty \br{\e \norm{\eta_0^{(\ell )} - \eta_{0,m}^{(\ell )} }^{2 + \psi} }^{1/\kappa} < \infty ~~~~~~~  ~(\ell=1,2)~,
        $$
        where $\eta_{j,m}^{(\ell )}$ is defined by
        \begin{align} \label{hd10}
            \eta_{j,m}^{(\ell )}  = f_\ell (\varepsilon_j , ..., \varepsilon_{j-m+1} , \mathbf{\varepsilon}^*_{j,m} )~~~~~~~  ~(\ell=1,2),
        \end{align}
        with $\mathbf{\varepsilon}^*_{j,m}  = (\varepsilon^*_{j,m,j-m} , \varepsilon^*_{j,m,j-m-1} , ...),$ and  $\varepsilon^*_{j,m,k}$ are i.i.d copies of $\varepsilon_0$ and independent of $\cbr{\varepsilon_j}_{j \in \mathbb{Z}}$.
    \end{enumerate}
\end{ass}

Note that our  assumptions are different from  those in \cite{gromenko}, who considered an  independent and identically distributed error process $\{\eta_j \}_{j\in \mathbb{Z}}$. 
In particular, we allow for different long-run variances before and after the change point $\lfloor n \vartheta_0 \rfloor$.
Moreover, these authors  postulate  separability in the spatiotemporal variance structure (in our case    a long-run variance),
which means that it factors into a purely spatial and a purely temporal component. This assumption simplifies the definition and the asymptotic analysis of their
test statistics substantially. 
We will demonstrate 
below that, by using the concept of self-normalization, we can construct 
(asymptotically) pivotal test statistics  for relevant hypotheses without any of these assumptions.

\section{Fully functional detection of relevant change points} 
\label{sec3} 
  \def\theequation{3.\arabic{equation}}
\setcounter{equation}{0}

We first consider a fully functional  approach for testing the relevant hypotheses in \eqref{hd2}.
As in the case of the classical hypothesis
$H_0: \|\delta \| =  0$,
it is based on the CUSUM 
statistic, but it turns out that for relevant hypotheses it will be more difficult to obtain asymptotic quantiles of a corresponding test statistic.
To be precise, we 
consider  the common  estimator for the unknown change point $\vartheta_0$   \citep[see][among many others]{hariz2007,jandhyala2013}
defined by
\begin{align}\label{def_changepoint_estimator}
    \hat{\vartheta}_n := \frac{1}{n} \argmax_{ \gbr{n \varepsilon} +1 \leq k \leq n - \gbr{n \varepsilon}}
    \frac{k(n-k)}{n^2} \norm{\frac{1}{k} \sum_{i=1}^k X_i - \frac{1}{n-k} \sum_{i=k+1}^n X_i }^2  ~,
\end{align}
where $\varepsilon >0$ is a small predefined constant.
It can be shown by similar arguments 
 as in Proposition 3.1 of  \cite{dette}
 that, under Assumption \ref{ass_2nd}, the estimator  
 $ \hat{\vartheta}_n$ is consistent,
 whenever $\norm{\delta}^2 >0$ and 
 $\vartheta_0 \in (\varepsilon, 1- \varepsilon)$,
 that is 
 \begin{align}
    \label{prop:changepoint_conv_rate}\hat{\vartheta}_n = \vartheta_0 + o_\p (n^{-1/2})~
     \end{align}
     as $n \to \infty$. 
Next, we define for $\lambda \in [0,1]$, $\vartheta \in (\varepsilon, 1- \varepsilon) $ the
 quantity
\begin{align} \label{hd3}
    D_n ( \lambda, \vartheta) := \frac{1}{\gbr{n \vartheta}} \sum_{j=1}^{\gbr{\lambda \gbr{n \vartheta}}} X_j - \frac{1}{n - \gbr{n \vartheta}} \sum_{j=\gbr{n \vartheta} + 1}^{\gbr{n \vartheta} + \gbr{\lambda (n - \gbr{n \vartheta})}} X_j ~~ \in L^2 (S \times [0,1])~,
\end{align}
where we also use the notation $D_n ( s,t, \lambda, \vartheta)$ simultaneously 
to make the 
dependence on the spatial and temporal component explicit. 
Note that   $D_n$ is well defined if $\varepsilon \geq \frac{1}{n}$, and we will assume throughout this paper that $n$ is sufficiently large such that this condition is satisfied.
If  $\vartheta =k/n$ and $\lambda =1 $  the quantity 
$ D_n ( 1, k/n) $   coincides with the expression in the squared 
norm in \eqref{def_changepoint_estimator}. Therefore,  $D_n ( 1, \hat{\vartheta}_n ) $ is a natural   estimator of the function $\delta$, which defines the difference before and after the change point. Consequently, it is reasonable to reject the null hypothesis in \eqref{hd2} for large values of the statistic
$$
\| D_n ( 1, \hat{\vartheta}_n ) \|.
$$
It will be shown later that 
$\sqrt{n} \big ( \| D_n ( 1, \hat{\vartheta}_n ) \|^2 - \| \delta \|^2 \big ) $ converges weakly to a 
normal distribution with a complicated variance depending on a linear combination of the long-run variances of the processes    $\{ \eta_j^{(1)} \}_{j \in \mathbb{Z}} $ and  $\{ \eta_j^{(2)} \}_{j \in \mathbb{Z}} $. In order to avoid its estimation, we will construct a pivotal statistic. Our main tool for this construction is the following result, which provides the weak convergence of the process 
$ \big   
\{\sqrt{n} ( \| D_n (\lambda, \hat{\vartheta}_n) \|^2  - \lambda^2 \| \delta \|^2 ) \big  \}_{\lambda \in [0,1]}$.
For its statement we denote by 
 \begin{align*} 
        K_\ell ((s_1, t_1), (s_2, t_2)) := \sum_{h \in \mathbb{Z}} \cov (\eta^{(\ell)}_0 (s_1, t_1), \eta^{(\ell)}_h (s_2, t_2))~~
    \end{align*}
the long-run covariance kernel of the process 
  $\{\eta^{(\ell)}_j\}_{j \in \mathbb{Z}}$ ($\ell=1,2$),
which exists under Assumption \ref{ass_2nd}, and by 
  \begin{align} 
  \label{hd5}
\bar{K}((s_1, t_1), (s_2, t_2)) :=  \tfrac{1}{\vartheta_0} K_1 ((s_1, t_1) , (s_2, t_2))) + \tfrac{1}{1-\vartheta_0} K_2 ((s_1, t_1) , (s_2, t_2)).      
    \end{align} 
a scaled  convex combination of these kernels.

\begin{thm}\label{thm:cp_spatiotemporal_mean}
{\it  If Assumption \ref{ass_2nd} is satisfied and $\| \delta \| > 0$, then
 \begin{align}
 \label{hd6a}
\sqrt{n} \big  \{ \|D_n (\lambda, \vartheta_0) \|^2 - \lambda^2 \norm{\delta}^2 
\big  \}_{\lambda \in \sqbr{0,1}} \convproc \tau_{\delta, \vartheta_0} & \cbr{\lambda \mathbb{B}(\lambda)}_{\lambda \in \sqbr{0,1}}, \\   
 \sqrt{n} \big \{ \| D_n (\lambda, \hat{\vartheta}_n)\|^2  - \lambda^2 \norm{\delta}^2  \big \}_{\lambda \in \sqbr{0,1}} \convproc \tau_{\delta, \vartheta_0} & \cbr{\lambda \mathbb{B}(\lambda)}_{\lambda \in \sqbr{0,1}}
 \label{hd6b}
 \end{align}
 as $n \to \infty$, where 
  $\convproc$ denotes weak convergence in $\ell^\infty (\sqbr{0,1})$, 
 $\cbr{ \mathbb{B}(\lambda)}_{\lambda \in \sqbr{0,1}}$ denotes a standard Brownian motion, 
    \begin{align}
    \label{hol1}
        \tau_{\delta, \vartheta_0}^2 := 4 \sum_{s_1, s_2 \in S} \int \negthickspace \int \delta(s_1, t_1) \delta(s_2, t_2) \bar  K ((s_1, t_1), (s_2, t_2)))  \mathrm{d}t_1 \mathrm{d}t_2,
    \end{align}
   and $\bar K$ is defined in \eqref{hd5}.
   }
\end{thm}

\noindent This result leads to a very simple and pivotal test for the relevant hypotheses in \eqref{hd2}. To be precise 
 we define
\begin{align} \nonumber 
    \hat{\mathbb{D}}_n &:= \|D_n (1, \hat{\vartheta}_n)\|^2,\\
  \label{det1a}
    \hat{\mathbb{V}}_n &:= \br{\int \br{\|D_n (\lambda, \hat{\vartheta}_n)\|^2 - \lambda^2 \|D_n (1, \hat{\vartheta}_n)\|^2}^2 \mathrm{d} \nu (\lambda) }^{1/2},
\end{align}
where $\hat{\vartheta}_n$ is defined as in (\ref{def_changepoint_estimator}) and $\nu$ a probability measure on the interval  $(0,1)$. We propose to reject  the null hypothesis in \eqref{hd2}, whenever
\begin{align} \label{hd7}
    \hat{\mathbb{D}}_n > \Delta + q_{1-\alpha} (\mathbb{W}) \hat{\mathbb{V}}_n,
\end{align}
where $q_{1-\alpha} (\mathbb{W})$ is the $(1-\alpha)$-quantile of the random variable
\begin{align}\label{disp:random_variable_w}
    \mathbb{W} := \frac{\mathbb{B} (1)}{\br{\int \lambda^2 (\mathbb{B}(\lambda) - \lambda \mathbb{B} (1) )^2 \mathrm{d}\nu (\lambda) }^{1/2}}.
\end{align}

The following result shows 
that the decision rule \eqref{hd7} defines a
  consistent and  asymptotic level $\alpha$ test for the hypotheses \eqref{hd2a}.

\begin{thm}\label{thm:test_asymptotics}
  {\it   If Assumption \ref{ass_2nd} is satisfied,
  $\alpha \leq 0.5$, $\Delta > 0$ and $\vartheta_0 \in \br{\varepsilon, 1- \varepsilon}$, we have 
    \begin{align*}
        \lim_{n \to \infty} \P{\hat{\mathbb{D}}_n > \Delta + q_{1-\alpha} (\mathbb{W}) \hat{\mathbb{V}}_n} = \begin{cases}
            0,& \text{if } \norm{\delta}^2 < \Delta,\\
            \alpha,& \text{if } \norm{\delta}^2 = \Delta  \text{ and } \tau_{\delta, \vartheta_0}^2> 0 ,\\
            1,& \text{if } \norm{\delta}^2 > \Delta.
        \end{cases}
    \end{align*}
    }
\end{thm}

\begin{rem} \label{remsym}
The distribution of the random variable  $\mathbb{W}$ in \eqref{disp:random_variable_w} 
is symmetric.
To see this, note that the numerator and denominator of $\mathbb{W}$ are independent. This follows  using the $L^2$ representation $\mathbb{B}(t) = t Z_0 + \sum_{k\geq 1} Z_k \frac{\sin (\pi k t)}{\pi k}$ of the Brownian motion 
and comparing $\mathbb{B}(1)$ and $\mathbb{B}(t) - t \mathbb{B}(1)$. Lastly, since $\mathbb{B}(1)$ is symmetric, the claim follows, because 
$$
\Big (-\mathbb{B}(1), (\int \lambda^2 (\mathbb{B}(\lambda) - \lambda \mathbb{B}(1))^2 \mathrm{d} \nu (\lambda))^{-1/2} \Big) \eqo{$\mathpcal{D}$} \Big (\mathbb{B}(1), (\int \lambda^2 (\mathbb{B}(\lambda) - \lambda \mathbb{B}(1))^2 \mathrm{d} \nu (\lambda))^{-1/2}
\Big ).
$$
 \end{rem}

\begin{rem} \label{remarkh0} ~~~
\begin{itemize}
    \item[(1)]
    Note that the test \eqref{hd7} depends on the specification of the measure $\nu$
   on the interval $[0,1]$, which has to be chosen in advance by the data analyst. However, in numerical experiments it turned 
   out that this dependence does not have a significant influence on the rejection probabilities, if the support of the measures has some distance to the boundaries $0$ and $1$ of this interval
(see Section \ref{sec5} for some results). 
A heuristic explanation for this observation 
consists in the fact  that the measure $\nu$
 appears in the definition 
of the statistic $\hat{\mathbb{V}}_n$ 
in \eqref{det1a} and in the quantiles of the random variable
$\mathbb {W}$ in \eqref{disp:random_variable_w}. Thus, intuitively, there is a cancellation effect in the decision rule \eqref{hd7}. 
  \item[(2)]  
  An important problem in applications is 
  the choice of the threshold $\Delta$, which is 
  problem specific. For this choice a careful discussion with experts from the field of application is recommended to understand in which difference they are really interested. 
  Moreover, there are also several alternatives, if this choice is difficult after these discussions. 
  \begin{itemize}
      \item[(a)]  It follows from the proof of Theorem \ref{thm:cp_spatiotemporal_mean} that 
      $$
\frac{ \hat{\mathbb{D}}_n - \| \delta \|^2}{
    \hat{\mathbb{V}}_n}
    \convd \mathbb{W}~.
      $$
      Consequently,
      an   (asymptotic) $(1-\alpha)$ confidence interval for the squared 
 norm  $\|\delta \|^2 \geq 0  $ of the difference of the mean functions before and after the change point is given by
\begin{equation}
\label{one}
\big [0,  \hat{\mathbb{D}}_n
+q_{1-\alpha}(\mathbb{W} )  \hat{ \mathbb{V}}_n \big ].
\end{equation} 
Similarly, if it can be ruled out that 
the squared  norm  vanishes, a two sided interval  for 
 $\|\delta \|^2 >0  $
 is given by 
 \begin{equation} 
 \label{two}
 \big ( \max \{0, \hat{\mathbb{D}}_n - q_{1-\alpha/2}(\mathbb{W})\hat{ \mathbb{V}}_n \},
 \hat{\mathbb{D}}_n
+q_{1-\alpha/2}(\mathbb{W} )  \hat{ \mathbb{V}}_n \big ].
\end{equation} 
 \item[(b)] It is also possible to test 
 the relevant hypotheses in \eqref{hd2}
 for a finite number of thresholds $\Delta^{(1)} < \ldots < \Delta^{(L)}$ simultaneously. 
 In particular, acceptance of  the null hypothesis
 with the threshold  $ \Delta^{(L_0)} $
 implies also acceptance for all 
  thresholds $ \Delta^{(L_0+1)} , \ldots  , \Delta^{(L)} $.
  Correspondingly, rejection for a $ \Delta^{(L_0)} $  means rejection for all smaller thresholds. In this sense, evaluating the test for several thresholds is  logically consistent for the user, and it is possible to determine for fixed level $\alpha$   the largest threshold such that the null hypotheses is rejected.
  \end{itemize}

\end{itemize}
\end{rem}

\section{Functional principal component analysis} 
\label{sec4} 

  \def\theequation{4.\arabic{equation}}
\setcounter{equation}{0}

In this section, we address the 
problem of detecting relevant changes in the mean of a stationary functional time series  by estimating  scores. For functional data this approach has been successfully used  by several authors in the context of testing ``classical''  hypotheses in the one-, two-sample and change point problem
 \citep[see][among  others]{benkohaerdle2009,berkes2009,zhangshao2015}, and it has been generalized to spatiotemporal data by \cite{gromenko}.
 To the best knowledge of the authors tests for relevant hypotheses have not been constructed by this approach.
 
To be precise, consider  model \eqref{hd1} and note that in this scenario, it is possible that the covariance function also changes at the point $\lfloor n \vartheta_0 \rfloor$ point.
Therefore, we denote by
$c^{(1)}$ and $c^{(2)}$ 
the covariance kernels corresponding to the samples
$ X_1, \ldots  , X_{\lfloor n \vartheta_0 \rfloor}$ and $X_{\lfloor n \vartheta_0 \rfloor +1 }, \ldots  , X_n $ before and after the change point, respectively, 
and define 
\begin{equation}
    \label{hd8a}
c_{\vartheta_0} := \vartheta_0 c^{(1)} + (1-\vartheta_0) c^{(2)} 
\end{equation}
as a convex combination of 
these two kernels.  We denote by  
$\tau_1 \geq \tau_2 \geq \ldots $ the 
ordered eigenvalues of the operator having covariance kernel $c_{\vartheta_0}$ with corresponding orthonormal eigenfunctions $w_1, w_2, \ldots $ in $L^2(S \times [0,1])$.
For a fixed integer $d\in \mathbb{N}$ we are interested in testing the (relevant) hypotheses
\begin{equation}
    \label{hd9}
    H_0: \sum_{k=1}^d \innpr{\delta, w_k}^2 \leq \Delta~~{\rm versus ~~}  H_1: \sum_{k=1}^d \innpr{\delta, w_k}^2 > \Delta~,
\end{equation}
where $\Delta >0 $ is a predefined threshold. Note that  by Parseval's identity
$\|\delta \|^2  = \sum_{k=1}^\infty \innpr{\delta, w_k}^2$, and therefore - similar as for testing classical hypotheses - a test for  the hypotheses \eqref{hd9} can also be used for the hypotheses \eqref{hd2}. We refer to Remark \ref{rem1} for a more detailed discussion of this approach in the context of  testing relevant hypotheses. For the  statements in this section we require the 
following assumptions.

\begin{ass}\label{ass_4th}
The process 
    $\cbr{X_j}_{j \in \mathbb{Z}}$ in 
    model \eqref{hd1} satisfies conditions (A1) and (A2) of Assumption \ref{ass_2nd}. Furthermore, $\cbr{X_j}_{j \in \mathbb{Z}}$ satisfies
    \begin{enumerate}
        \item[(A3')] There exists a constant  $\psi \in \br{0,1}$ such that $\e \norm{\eta_j^{(\ell)}}^{4 + \psi} < \infty$ ~~~~ $(\ell = 1,2)$.
        
        \item[(A4')] 
        The sequences $\{\eta_j^{(1 )}\}_{j \in \mathbb{Z}}$  and $\{\eta_j^{(2 )}\}_{j \in \mathbb{Z}}$ can be approximated by $m$-dependent sequences $\{\eta_{j,m}^{(1 )}\}_{j \in \mathbb{Z}}$  and $\{\eta_{j,m}^{(2 )}\}_{j \in \mathbb{Z}}$, respectively, 
         in the sense that for some $\kappa > 4 + \psi$ $$\sum_{m=1}^\infty \br{\e \norm{\eta_0^{(\ell )} - \eta_{0,m}^{(\ell )} }^{4 + \psi} }^{1/\kappa} < \infty ~~~~ (\ell =1,2),$$
        where $\eta_{j,m}^{(\ell )}$ is defined \eqref{hd10}.
    \end{enumerate}
\end{ass}

\begin{ass}\label{eigvals_order}
The (ordered) eigenvalues of the covariance operator  $c_{\vartheta_0}$ in \eqref{hd8a} satisfy 
    $\tau_1 > \cdots > \tau_d > \tau_{d+1} > 0$, where  $d \in \mathbb{N}$ is the number of scores considered in \eqref{hd9}.
\end{ass}

Recall the definition of the   estimator for the change point in \eqref{def_changepoint_estimator}
and for  $f,g \in L^2(S \times [0,1])$, we define the function $f \otimes g \in L^2((S \times [0,1])^2)$ as $(f \otimes g) ((s_1, t_1), (s_2, t_2)) := f(s_1, t_1) g(s_2, t_2)$. 
We consider 
\begin{align} \label{hd11}
    \hat{c}_{\hat{\vartheta}_n, \lambda} := \hat{\vartheta}_n \hat{c}^{(1)}_\lambda + (1- \hat{\vartheta}_n) \hat{c}^{(2)}_\lambda,
\end{align}
as a sequential estimator of the  convex combination \eqref{hd8a},
where
\begin{align*}
    \hat{c}^{(1)}_\lambda &:= \frac{1}{\gbr{\lambda \gbr{n \hat{\vartheta}_n}}} \sum_{i=1}^{\gbr{\lambda \gbr{n \hat{\vartheta}_n}}} \br{X_i - \frac{1}{\gbr{\lambda \gbr{n \hat{\vartheta}_n}}} \sum_{j=1}^{\gbr{\lambda \gbr{n \hat{\vartheta}_n}}} X_j} \otimes \br{X_i - \frac{1}{\gbr{\lambda \gbr{n \hat{\vartheta}_n}}} \sum_{j=1}^{\gbr{\lambda \gbr{n \hat{\vartheta}_n}}} X_j},\\
    \hat{c}^{(2)}_\lambda &:= \frac{1}{\gbr{\lambda (n-\gbr{n \hat{\vartheta}_n})}} \sum_{i=\gbr{n \hat{\vartheta}_n}+1}^{\gbr{n \hat{\vartheta}_n} + \gbr{\lambda (n-\gbr{n \hat{\vartheta}_n})}} \br{X_i - \frac{1}{\gbr{\lambda (n-\gbr{n \hat{\vartheta}_n})}} \sum_{j=\gbr{n \hat{\vartheta}_n}+1}^{\gbr{n \hat{\vartheta}_n} + \gbr{\lambda (n-\gbr{n \hat{\vartheta}_n})}} X_j} \\
    &\phantom{asdfasdfasdasdadsasassaadfaadsf} \otimes \br{X_i - \frac{1}{\gbr{\lambda (n-\gbr{n \hat{\vartheta}_n})}} \sum_{j=\gbr{n \hat{\vartheta}_n}+1}^{\gbr{n \hat{\vartheta}_n} + \gbr{\lambda (n-\gbr{n \hat{\vartheta}_n})}} X_j}
\end{align*}
are estimates of the covariance functions $c^{(1)}$ and $c^{(2)}$ before and after the change point, respectively. 
For $i=1,2$ we put $\hat{c}^{(i)}_\lambda := 0$, if $\gbr{\lambda \gbr{n \hat \vartheta_n}} = 0$. The first result of this section  shows that the statistic  \eqref{hd11} is a 
uniformly consistent  estimator for the convex combination in \eqref{hd8a}.
\begin{thm}
\label{thm7}
{\it 
   If Assumption \ref{ass_4th} is satisfied and $\vartheta_0 \in (\varepsilon, 1- \varepsilon)$, we have
    \begin{align*}
        \sup_{0 \leq \lambda \leq 1} \sqrt{\lambda} \norm{\hat{c}_{\hat{\vartheta}_n, \lambda} - c_{\vartheta_0}}_2 = O_\p \br{\frac{\log^{2/\kappa} (n)}{\sqrt{n}}},
    \end{align*}
    where $\norm{\cdot}_2$ denotes the norm induced by  
     the inner product 
\begin{align}
\label{det1}
    \innpr{f,g}_2 := \sum_{s_1,s_2 \in S} \int \negthickspace \int f((s_1, t_1), (s_2, t_2)) g((s_1, t_1), (s_2, t_2)) \mathrm{d}t_1 \mathrm{d}t_2
\end{align}
 on  $L^2((S \times [0,1])^2)$. 
 }
\end{thm}

In the following discussion, we  will denote the eigenfunctions and eigenvalues of the estimator $\hat{c}_{\hat{\vartheta}_n, \lambda}$ 
 by $\hat{w}_{1,\lambda}, \hat{w}_{2, \lambda}, ...$  and  $\hat{\tau}_1, \hat{\tau}_2, ...$, respectively.  Recall the definition of  the process 
 $D_n(\lambda, \vartheta ) $ in \eqref{hd3} and note that  $D_n(1, \hat{\vartheta}_n) $
 is an estimator of difference between the mean functions before and after the change point. Consequently, a natural estimator of the quantity 
 $ \sum_{k=1}^d \innpr{\delta, w_k}^2$  in  \eqref{hd9} is given by  the statistic
\begin{equation} \label{hd17}
    \sum_{k=1}^d \innpr{D_n(1, \hat{\vartheta}_n), \hat{w}_k}^2 ~,
\end{equation}
and the null hypothesis in \eqref{hd9} will be rejected for large values of this statistic. The following result 
shows that the process 
$\sqrt{n} \{\sum_{k=1}^d \big ( \innpr{D_n(\lambda , \hat{\vartheta}_n), \hat{w}_k}^2 - \lambda^2 \innpr{\delta, w_k}^2  \big )  \}_{\lambda \in [0,1]} $ converges 
weakly, and -  as a by-product - establishes asymptotic normality 
of \eqref{hd17} (after appropriate normalization).

\begin{thm} \label{thm9}
{\it    If Assumption \ref{ass_4th} is satisfied, $\norm{\delta}^2 > 0$ and $\vartheta_0 \in (\varepsilon, 1- \varepsilon)$, then the following statements are true $n \to \infty$.
    \begin{align}
    \label{hd18a}
   &  \sqrt{n} \bigg\{\sum_{k=1}^d \br{\innpr{D_n (\lambda, \vartheta_0), \hat{w}_{k, \lambda}}^2 - \lambda^2 \innpr{\delta, w_k}^2}\bigg\}_{\lambda \in \sqbr{0,1}} \convproc \sigma_{cp} \cbr{\lambda \mathbb{B}(\lambda)}_{\lambda \in \sqbr{0,1}}, \\
& \sqrt{n} \bigg\{\sum_{k=1}^d \br{\innpr{D_n (\lambda, \hat{\vartheta}_n), \hat{w}_{k, \lambda}}^2 - \lambda^2 \innpr{\delta, w_k}^2} \bigg\}_{\lambda \in \sqbr{0,1}} \convproc \sigma_{cp} \cbr{\lambda \mathbb{B}(\lambda)}_{\lambda \in \sqbr{0,1}}, \label{hd18b}
    \end{align}
where 
\begin{align*}
    \sigma_{cp}^2 := &2\sum_{i\in \mathbb{Z}} \bigg( \vartheta_0 \sum_{s_1,...,s_4 \in S} \int \cov \br{ \eta_0^{(1)} (t_1, s_1) \eta_0^{(1)} (t_2, s_2) , \eta_i^{(1)} (t_3, s_3) \eta_i^{(1)} (t_4, s_4)} \\
    & ~~~~~~~~~~~~~~~~~~~~~~~~~~~~~~~~
    ~~~~~~~~ ~~~~~~~~ \times 
    f((t_1, s_1), (t_2, s_2)) f((t_3, s_3), (t_4, s_4)) \mathrm{d} (t_1, ..., t_4) \\
    &- 2 \sum_{s_1, s_2, s_3 \in S} \int \cov \br{ \eta_0^{(1)} (t_1, s_1) \eta_0^{(1)} (t_2, s_2)  , \eta_i^{(1)} (t_3, s_3)  } f((t_1, s_1), (t_2, s_2)) w(t_3, s_3) \mathrm{d} (t_1, t_2, t_3)\\
    &+  (1- \vartheta_0) \negthickspace \negthickspace \sum_{s_1,...,s_4 \in S} \int \cov \br{ \eta_0^{(2)} (t_1, s_1) \eta_0^{(2)} (t_2, s_2) , \eta_i^{(2)} (t_3, s_3) \eta_i^{(2)} (t_4, s_4)}\\
    & ~~~~~~~~~~~~~~~~~~~~~~~~~~~~~~~~
    ~~~~~~~~ ~~~~~~~~ \times 
    f((t_1, s_1), (t_2, s_2)) f((t_3, s_3), (t_4, s_4)) \mathrm{d} (t_1, ..., t_4) \\
    &-2 \sum_{s_1, s_2, s_3 \in S} \int \cov \br{ \eta_0^{(2)} (t_1, s_1) \eta_0^{(2)} (t_2, s_2)  , \eta_i^{(2)} (t_3, s_3)  } f((t_1, s_1), (t_2, s_2)) w(t_3, s_3) \mathrm{d} (t_1, t_2, t_3) \bigg)\\
    &+ \sum_{s_1, s_2 \in S} \int \bar K ((t_1, s_1), (t_2, s_2)) w(t_1, s_1) w(t_2, s_2) \mathrm{d} (t_1, t_2)
\end{align*}
    with $\bar K$ defined in \eqref{hd5} and
    \begin{align} \label{h1}
        w(s_1,t_1) &:= \sum_{k=1}^d \innpr{\delta, w_k} w_k (s_1, t_1),\\ \ppad f((s_1, t_1), (s_2, t_2)) &:= \sum_{k=1}^d \innpr{\delta, w_k} w_k (s_1, t_1) \sum_{j \neq k} \frac{\innpr{\delta, w_j}}{\tau_k - \tau_j} w_j (s_2, t_2).
        \label{h2}
    \end{align}
    }
\end{thm}

Note that Theorem \ref{thm9} implies,
 the statistic 
$\sqrt{n}\sum_{k=1}^d \big ( \innpr{D_n(1, \hat{\vartheta}_n), \hat{w}_k}^2 - \innpr{\delta, w_k}^2  \big )$ converges weakly to a normal distribution with mean $0$ and variance $\sigma_{cp}^2$. However, 
the limiting variance $ \sigma_{cp}^2$ 
is rather difficult to estimate and therefore we propose again to proceed by self-normalization.
For this purpose define 
\begin{equation} \label{det32}
\hat{\mathbb{Y}}_n := \br{\int_0^1 \br{\sum_{k=1}^d  \innpr{ D_n (\lambda, \hat{\vartheta}_n) , \hat{w}_{k, \lambda} }^2 - \lambda^2 \innpr{D_n (1, \hat{\vartheta}_n), \hat{w}_k}^2  }^2 \mathrm{d} \nu (\lambda) }^{1/2},
\end{equation}
for some probability measure $\nu$ on the interval $(0,1)$. Then we propose to reject 
the null hypothesis in 
\eqref{hd9}, whenever
\begin{align}\label{disp:cp_mean_testrule}
    \sum_{k=1}^d \innpr{D_n(1, \hat{\vartheta}_n), \hat{w}_k}^2 > \Delta + q_{1-\alpha} (\mathbb{W}) \hat{\mathbb{Y}}_n.
\end{align}
The next result shows that  this decision rule defines a reasonable test for the hypotheses \eqref{hd9}. The proof is obtained by similar arguments  as  
given in the proof of Theorem 
\ref{thm:test_asymptotics} and is therefore omitted.
\begin{thm}\label{theorem6}
{\it 
If Assumption \ref{ass_4th} is satisfied, $\alpha \leq 0.5$, $\Delta > 0$ and $\vartheta_0 \in  (\varepsilon, 1- \varepsilon)$,
we have
    \begin{align*}
        \lim_{n \to \infty} \P{\sum_{k=1}^d \innpr{D_n(1, \hat{\vartheta}_n), \hat{w}_k}^2 > \Delta + q_{1-\alpha} (\mathbb{W}) \hat{\mathbb{Y}}_n} = \begin{cases}
            0,& \text{if } \sum_{k=1}^d \innpr{\delta, w_k}^2 < \Delta,\\
            \alpha,& \text{if } \sum_{k=1}^d \innpr{\delta, w_k}^2 = \Delta \text{ and }  \sigma_{cp}^2 >0 ,\\
            1,& \text{if } \sum_{k=1}^d \innpr{\delta, w_k}^2 > \Delta.
        \end{cases}
    \end{align*}
    }
\end{thm}

\begin{rem} \label{rem1}
Parseval's implies 
$\|\delta \|^2  = \sum_{k=1}^\infty \innpr{\delta, w_k}^2$, and one can  also use the decision rule \eqref{disp:cp_mean_testrule}
for testing the hypotheses \eqref{hd2} (the null hypothesis is rejected, if \eqref{disp:cp_mean_testrule} holds). As a consequence of Theorem \ref{thm9}, we obtain 
for the rejection probabilities of this test - provided that all the requirements stated in Theorem \ref{theorem6} are satisfied
     \begin{align*}
        \lim_{n \to \infty} \P{\sum_{k=1}^d \innpr{D_n(1, \hat{\vartheta}_n), \hat{w}_k}^2 > \Delta + q_{1-\alpha} (\mathbb{W}) \hat{\mathbb{Y}}_n} = \begin{cases}
            0,& \text{if } \norm{\delta}^2 < \Delta,\\
            1,& \text{if } \sum_{k=1}^d \innpr{\delta, w_k}^2 > \Delta
        \end{cases}
    \end{align*}
    and, if $\norm{\delta}^2 = \Delta$, and $\sigma_{cp}^2 > 0$, we have
    \begin{align*}
        \lim_{n \to \infty} \P{\sum_{k=1}^d \innpr{D_n(1, \hat{\vartheta}_n), \hat{w}_k}^2 > \Delta + q_{1-\alpha} (\mathbb{W}) \hat{\mathbb{Y}}_n} \leq \alpha.
    \end{align*}
    This means that the decision rule (\ref{disp:cp_mean_testrule}) defines  also a  (conservative)  asymptotic level $\alpha$ test for the hypotheses \eqref{hd2}. Moreover, similar as in the case of testing classical hypotheses the test is consistent, whenever $\sum_{k=1}^d \innpr{\delta, w_k}^2 > \Delta$.
\end{rem}

\section{Finite sample properties} 
\label{sec5} 
  \def\theequation{5.\arabic{equation}}
\setcounter{equation}{0}

In this section, we illustrate the finite sample properties of the proposed tests by means of a small simulation study and by the analysis of a data example. 
Throughout this section, if not mentioned otherwise, we use a uniform distribution 
$\nu_{19} = \frac{1}{19} \sum_{i=1}^{19} \delta_{i/20}$
at the the points $1/19, \ldots , 18/19$ for the 
the probability measure $\nu$ in the pivotal statistic (here $\delta_x$ denotes the Dirac measure
at the point $x$).  We also demonstrate below that 
the tests \eqref{hd7} and \eqref{disp:cp_mean_testrule} are not very sensitive with respect to the choice of this measure.

\subsection{Simulation study} \label{sec51} 
For the illustration of the methods introduced
 in Section \ref{sec3} (fully functional) 
and in Section \ref{sec4} (functional principal components) we put $\mu \equiv 0$ and as difference before and after the change point we use the function 
\begin{align}\label{delta_function}
    \delta(s,t) = \sqrt{\gamma} s \cos \br{\frac{\pi}{2} t}.
\end{align}
The parameter $\gamma$ will be used to vary the size of the quantity $\norm{\delta}^2$ and 
$\sum_{k=1}^d \innpr{\delta, v_k}^2$
in the hypotheses \eqref{hd2} and \eqref{hd9}, respectively.
In all cases we consider $4$ locations, that is  $S=\cbr{1,2,3,4}$, the threshold is $\Delta = 0.15$ and 
the sample size is given by  $n= 150, 250, 500$. The position of the change point is  chosen as $\vartheta_0 = 0.6$ and the tuning parameter $\varepsilon$ in the change point estimator defined in \eqref{def_changepoint_estimator} is set to $0.05$. 
Throughout this section all results are based on $1000$ simulation runs.

For both procedures we consider  different error processes $\{\eta_i \}_{i \in \mathbb{Z}}$
in model \eqref{hd1}, where we assume that the processes before and after the change point 
have the same distribution, that is 
$\{ \eta_i \}_{i \in \mathbb{Z}} \overset{{\cal D}}{=} \{ \eta_i^{(1)} \}_{i \in \mathbb{Z}}\overset{{\cal D}}{=}
\{ \eta_i^{(2)} \}_{i \in \mathbb{Z}} $. The first one consists of 
independent (scaled) Brownian motions, that is
\begin{equation}
    \label{det1b}
      \eta_i (s,t) = {s \over 4}  \mathbb{W}_i(t) ~,~i=1, \ldots , n,
\end{equation}
(here $\mathbb{W}_i$ denotes the standard Brownian motion on the interval $[0,1]$). Note that this process has a separable covariance function. 
Secondly, we consider a  process 
$\{  \eta_i \}_{i\in \mathbb{Z}}$ 
of independent functions with non-separable
covariance  defined by 
\begin{align}
\label{det1b1}
    \eta_i (s,t) = \sum_{k=1}^\infty \frac{N_{k,i}}{2 \pi k} (\sin (2 \pi k t) + s \cdot \cos (2 \pi k t) ),
\end{align}
where $N_{k,i}$ denote independent standard normal distributed random variables. 
The third process is a functional moving average 
process (fMA) of order $1$. More precisely, we 
define independent processes 
\begin{align}
    \label{det1d1}
    \varepsilon_i (s,t) = \sum_{j=1}^\infty N_{ji} \sqrt{\lambda_j} v_j (s,t),
\end{align}
where $N_{ji}$ are independent standard normal
distributed 
random variables, $\lambda_j = (2 \pi j^2)^{-1}$ and $v_j (s,t) = \frac{\sqrt{2}}{4} \cdot s \cdot \sin (2 j \pi t)$, and consider the fMA(1) process
\begin{align}
    \label{det1d}
\eta_i (s,t) = \varepsilon_i + 0.7 \varepsilon_{i-1}~.
\end{align}
(we use  the first $40$ terms in  the expansion  \eqref{det1d1}). The data is generated and stored in Fourier basis representations using the R-package \texttt{fda}.

\subsubsection{Fully functional detection of relevant changes} 
\label{sec511}
We begin with an investigation of the fully functional 
approach and display in  Figure \ref{fig1} the  rejection probabilities of the test (\ref{hd7}) for the hypotheses (\ref{hd2}) (with $\Delta = 0.15$) for the three different error processes. 
We observe a qualitatively similar behaviour in all three cases as predicted by Theorem \ref{thm:test_asymptotics}. The rejection probability is strictly increasing with $\| \delta \|^2$.
It is close to the nominal level $5\%$ if $\| \delta \|^2 = \Delta $ and smaller (larger) 
than $5\%$  if $\| \delta \|^2 <\Delta $ ($\| \delta \|^2 > \Delta $). 
A comparison of the three error processes shows 
that the best power of 
the test is obtained for  the error process
\eqref{det1b1} followed  by the 
error process  \eqref{det1d},  while for the 
error process \eqref{det1b} the test is less powerful.

\begin{figure}[!ht]
    \centering
     {{\includegraphics[width=5.4cm]{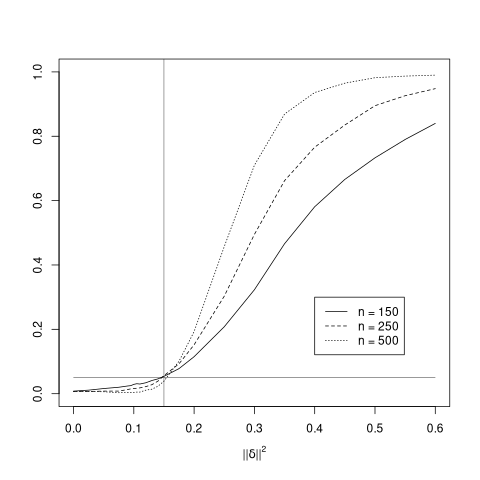} }}%
     {{\includegraphics[width=5.4cm]{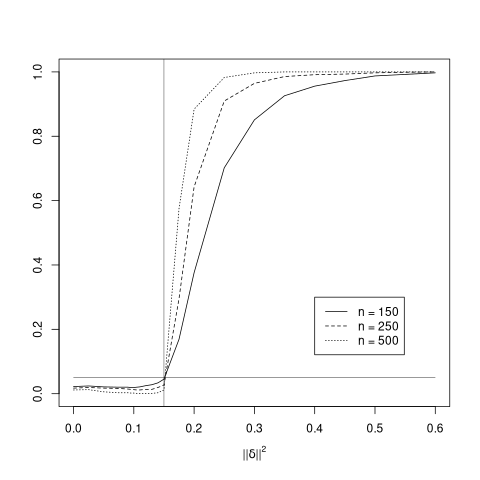} }}%
     {{\includegraphics[width=5.4cm]{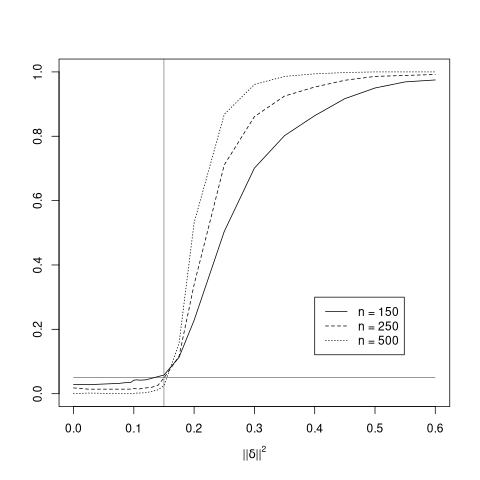} }}%
    \caption{\it
    Empirical rejection probabilities of the test (\ref{hd7}) for the hypotheses (\ref{hd2})
    with $\Delta = 0.15$.  The difference $\delta$ between the mean functions   is given by (\ref{delta_function}) and different error processes are considered.
    Left panel:  scaled Brownian motion  \eqref{det1b};  middle panel: non-seperable process  in \eqref{det1b1};
    right panel: fMA(1) process in \eqref{det1d}. }
    \label{fig1}
\end{figure}

\begin{figure}[!ht]
    \centering
     {{\includegraphics[width= 5.4cm]{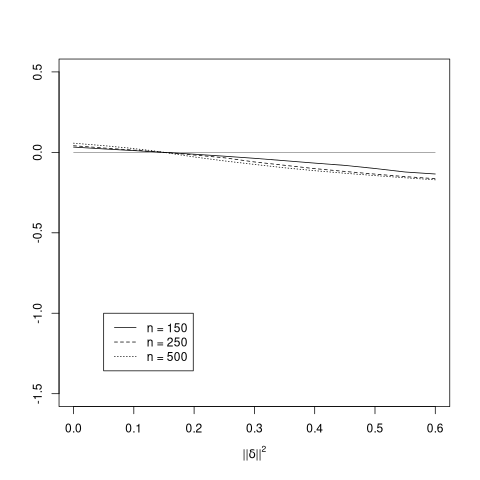} }}%
     {{\includegraphics[width= 5.4cm]{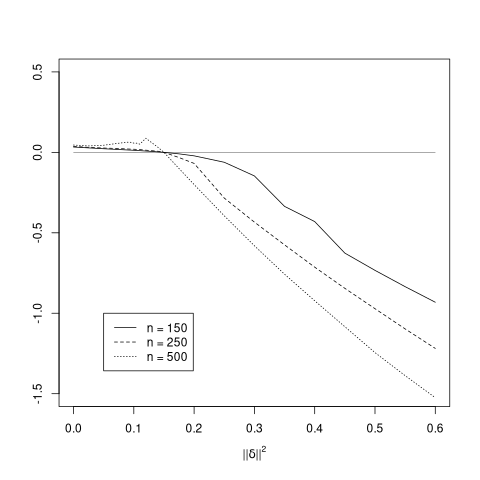} }}%
     {{\includegraphics[width= 5.4cm]{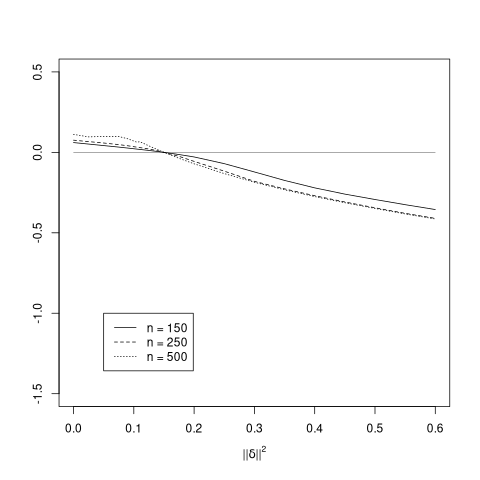} }}%
    \caption{ \it Size of the term $(\Delta - \norm{\delta}^2)/   \tau_{\delta, \vartheta_0} $
     as a function $\norm{\delta}^2$
     in the approximation of the rejection probability  ($\Delta = 0.15$) in \eqref{hol2}.
      Left panel:  scaled Brownian motion  \eqref{det1b};  middle panel: non-seperable process  in \eqref{det1b1};
    right panel: fMA(1) process in \eqref{det1d}. }
\end{figure}

This  observation can be explained  by  the fact that the different error  processes have different variability.
More precisely, it follows from the proof of 
Theorem \ref{thm:test_asymptotics} that 
for large sample sizes the probability of rejection can be approximated by 
 \begin{align}
\label{hol2}
     \P{\hat{\mathbb{W}} > {\sqrt{n}  \over \mathbb{V}}{
     \Delta - \norm{\delta}^2 
     \over    \tau_{\delta, \vartheta_0} } 
     + q_{1-\alpha} (\mathbb{W}) } ~, 
 \end{align}
where  $  \tau_{\delta, \vartheta_0}^2 $
is defined in \eqref{hol1} and 
$\mathbb{V}= \br{\int \lambda^2 (\mathbb{B}(\lambda) - \lambda \mathbb{B} (1) )^2 \mathrm{d}\nu (\lambda) }^{1/2}$
denotes a generic random variable (a
functional of the Brownian motion). Thus, the power is dominated by the term 
$(\Delta - \norm{\delta}^2)/ 
      \tau_{\delta, \vartheta_0} $, which is negative under the alternative.  A  smaller value of this term  
      results  in a larger power  and in Figure \ref{fig1} we display this quantity 
      as a function of $\|\delta \|^2$. The 
      results explain  the differences in  the simulated power for the three processes under consideration.
      \\
We also note that the approximation of the nominal
level at  the boundary  of the hypotheses $\|\delta\|^2 = \Delta$
differs in the scenarios. For small sample sizes it is  more accurate for the non-separable process
\eqref{det1b1} compared to the two other cases.
A possible explanation for this observation
is the different   
accuracy of the change point estimator \eqref{def_changepoint_estimator} in the three scenarios,
which is displayed in Figure \ref{cp_hist1}.
We observe that the change point estimator for the error process \eqref{det1b}
 exhibits a larger variability than in the two other cases, while the smallest variability is obtained for the error process \eqref{det1b1}.

\begin{figure}[!ht]
    \centering
         {{\includegraphics[width= 5.4cm]{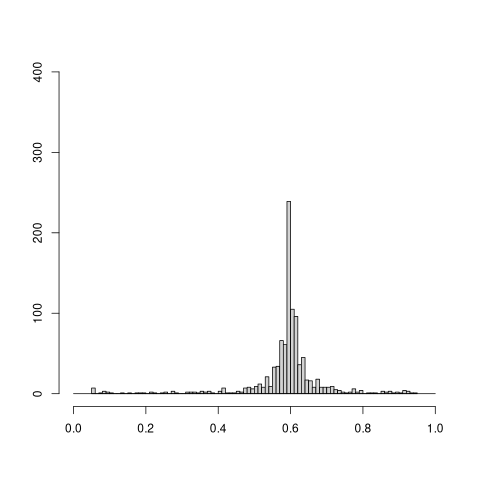} }}%
     {{\includegraphics[width= 5.4cm]{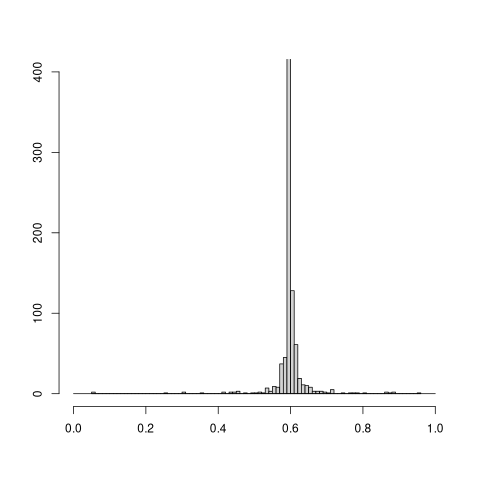} }}%
      {{\includegraphics[width= 5.4cm]{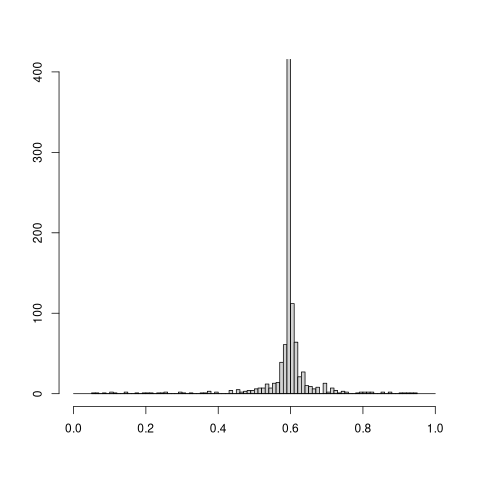} }}%
    \caption{ \it Histogram of the change point estimator based on $n=250$ observations. The difference between the  mean functions is given by \eqref{delta_function}, $\Delta = 0.15$, the change point is located as
    $\vartheta_0 = 0.6$ and different scenarios for the error process
    are considered. Left panel:  scaled Brownian motion in \eqref{det1b};  middle panel: non-separable process  in \eqref{det1b1};
    right panel: fMA(1) process in \eqref{det1d}.} 
    \label{cp_hist1}
\end{figure}

Next we investigate the impact of the measure  $\nu$ in the self-normalizing factor \eqref{det1a} on the properties of the test \eqref{hd7}. 
Note that this
measure appears in the definition 
of the statistic $\hat{\mathbb{V}}_n$ 
in \eqref{det1a} and in the random variable
$\mathbb {W}$ in \eqref{disp:random_variable_w}. Thus, intuitively, there is a cancellation effect in the decision rule \eqref{hd7}. 
For the sake of brevity, 
we restrict ourselves to the case of the fMA(1) error process
\eqref{det1d}  and display in Figure \ref{fig3}  the rejection probabilities of the test \eqref{hd7}, where we use
a uniform distribution
\begin{equation} \label{det31}
\nu_k = \frac{1}{k} \sum_{i=1}^{k} \delta_{i/(k+1)}
\end{equation} 
at $k=4,9$ and $k=19$ points as 
measure in the statistic \eqref{det1a}.
We observe a rather similar behaviour for all three measures, where $\nu_{19}$
yields a slightly better approximation of the nominal level at the boundary of the hypotheses 
(that is $\|\delta \|^2= \Delta = 0.15$) 
for the sample size $n=150$.

\begin{figure}[!ht]
    \centering
     {{\includegraphics[width= 5.4cm]{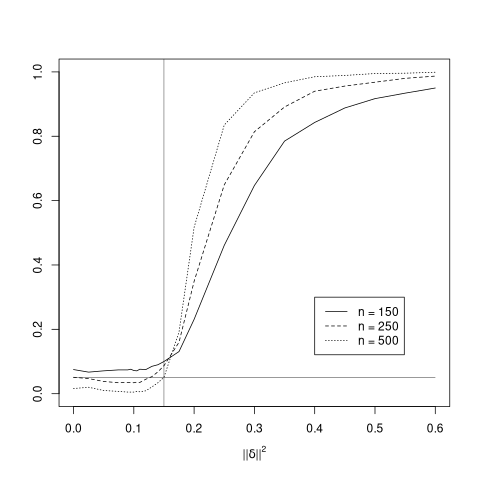} }}%
     {{\includegraphics[width= 5.4cm]{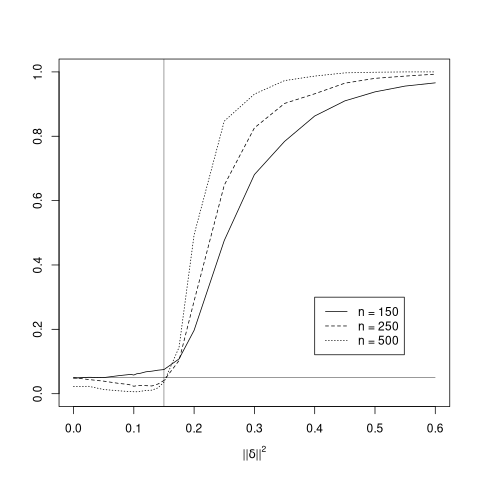} }}%
     {{\includegraphics[width= 5.4cm]{rejcurves/ma-norm} }}%
    \caption{
   \it 
    Empirical rejection probabilities of the test (\ref{hd7}) for the hypotheses (\ref{hd2})
    with $\Delta = 0.15$.  The error process is given by an fMA(1) model 
    and different measures in the statistic \eqref{det1a}  are considered.
    Left panel: $\nu_4$, middle panel: $\nu_9$, right panel: $\nu_{19}$.
    \label{fig3} }
\end{figure}

\begin{figure}[!ht]
    \centering
     {{\includegraphics[width= 5.4cm]{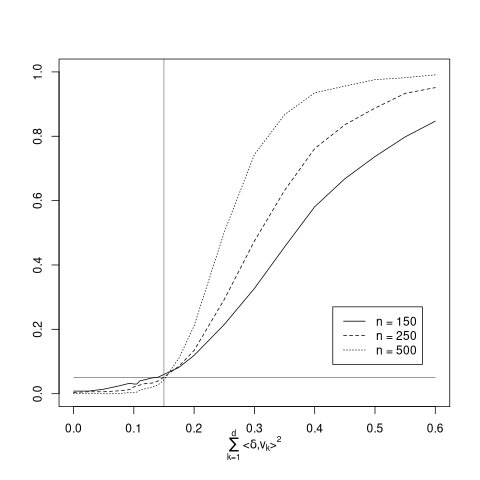} }}%
     {{\includegraphics[width= 5.4cm]{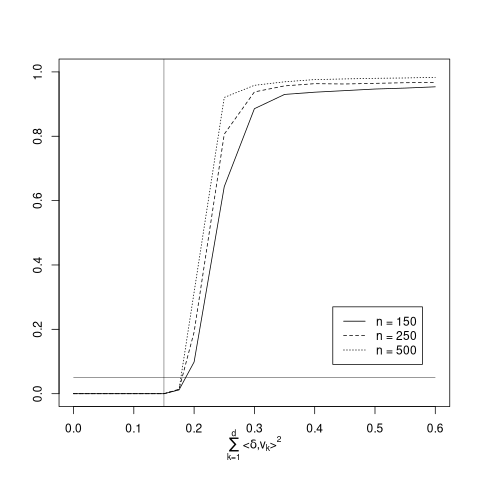} }}%
     {{\includegraphics[width= 5.4cm]{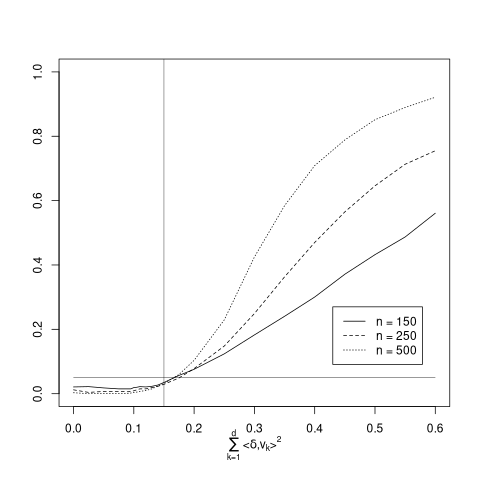} }}%
    \caption{ 
    \it 
    Empirical rejection probabilities of the test 
    (\ref{disp:cp_mean_testrule}) for the hypotheses (\ref{hd9})
    with $\Delta = 0.15$.  
    The difference $\delta$ between the mean functions 
    is given by (\ref{delta_function}) and different error processes are considered.
    Left panel:   Brownian motion  ($d=15$) ; middle panel:  non-separable process  ($d=11$); right panel: fMA(1) process  ($d=11$).}
        \label{fig2}
\end{figure}

\subsubsection{Relevant changes by 
functional principal components} 
\label{sec512}

In this section, we briefly  illustrate the finite sample properties of the test
(\ref{disp:cp_mean_testrule}) for the hypotheses (\ref{hd9}), where we use the same scenarios as before.  The test requires the choice of the number of principal functional components and we choose
 the parameter $d$ 
 such that 95\% of the variance in the data will be explained. This results in
 $d=15$, $d=11$ and $d=11$ functional principal components 
 for the models \eqref{det1b}, \eqref{det1b1} and \eqref{det1d},
 respectively. 
The corresponding rejection probabilities are displayed in Figure \ref{fig2} and we   observe the qualitative
behaviour predicted by Theorem \ref{theorem6}.
We also observe that the test  (\ref{disp:cp_mean_testrule}) 
is conservative in the case of the non-seperable process
\eqref{det1b1}, while the nominal level at the boundary of the hypotheses $\sum_{k=1}^d \innpr{\delta,  b_k}^2= \Delta$ is very well approximated for the Brownian motion \eqref{det1b}.


Finally we investigate the impact of the measure $\nu$ in the scaling factor \eqref{det32}, where 
we again restrict ourselves to the case of an fMA(1) process and the uniform distributions $\nu_k$ in \eqref{det31} for $k=4,9$ and $19$.
The corresponding results are shown in Figure \ref{fig6} and demonstrate that the test 
\eqref{disp:cp_mean_testrule} is not very sensitive with respect to this choice.

\begin{figure}[!ht]
    \centering
     {{\includegraphics[width= 5.4cm]{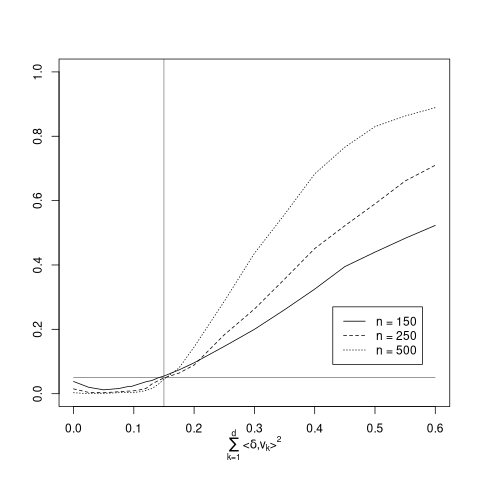} }}%
     {{\includegraphics[width= 5.4cm]{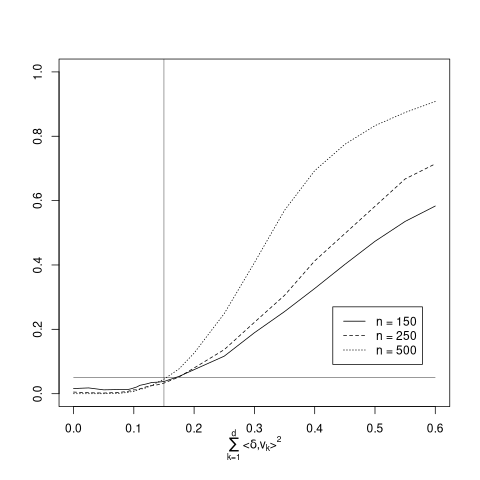} }}%
     {{\includegraphics[width= 5.4cm]{rejcurves/ma-d11} }}%
    \caption{
      \it
    Empirical rejection probabilities of the test (\ref{disp:cp_mean_testrule})
    for the hypotheses \eqref{hd9} 
    with $\Delta = 0.15$.  The error process is given by an fMA(1) model 
    and different measures in the statistic \eqref{det1a}  are considered.
    Left panel: $\nu_4$, middle panel: $\nu_9$, right panel: $\nu_{19}$.
    \label{fig6}}
\end{figure}

\subsection{Data example} \label{sec52}

We conclude this paper with an  application of 
the two test procedures in a real  data example. For this purpose, we use Canadian weather data, which consists of  daily measurements at $40$ representative Canadian cities.
Thus we observe yearly curves at different locations  over different years ($1916 - 2018$) at different locations. The data  can be downloaded from  the government of Canada website: \url{https://climate.weather.gc.ca/historical_data/search_historic_data_e.html}. The available data contains several different measurements such as maximum/minimum temperature or precipitation amount. Here, for the sake of brevity, we concentrate on  the average daily temperatures. Due to missing values in the reported temperature data, four  stations were chosen such that a large amount of available data overlaps and only little parts had to be interpolated or removed:  
 \textsc{Calgary International Airport, Alberta} (ID: 2205), \textsc{Medicine Hat Airport, Alberta} (ID: 2273), \textsc{Indian Head CDA, Saskatchewan} (ID: 2925) and \textsc{Ottawa CDA, Ontario} (ID: 4333). In the notation of the previous sections this means $S= \{1,2,3,4\}$ and the  sample size is given by $n=113$, which corresponds to the period from years $1891 -2007$ without the years $1910,$ $1911,$ $1993$ and $1995$.

The change point estimator in \eqref{def_changepoint_estimator} gives  $\hat{\vartheta}_n = 0.708$, which approximately corresponds to the year $1969$.  We first investigate the fully functional approach in Section \ref{sec3}. The results of the test 
\eqref{hd7} for different thresholds and different nominal  level are given in Table
\ref{tab1}. We observe that $\Delta = 2.974 $
is the largest threshold  such that the test \eqref{hd7}  rejects the null hypothesis at nominal level $\alpha =0.05$.
Because there are $4$  stations this corresponds to an average effect of $ {2.974 /4} \approx 0.74$.  Finally, we note that the one-sided confidence interval  $\|\delta \|  \geq 0  $ 
in \eqref{one} is given by $\sqbr{0, 4.397}$
while the  two-sided interval  $\|\delta \| >0  $ in  \eqref{two} 
is obtained as $\sqbr{0.672,4.675}$
\begin{table}[!ht]
    \centering
    \begin{tabular}{c | c | c | c}
         $\Delta$ & 10 \% & 5 \% & 1 \%  \\
         \hline
         2.974 & reject & reject & accept\\
         2.975 & reject & accept & accept
    \end{tabular}
    \caption{\it
    Results of the test \eqref{hd7} for the Canadian weather data for different nominal level and different thresholds. The value $2.974$ represents  the maximal threshold \eqref{hd2a}, such that the null hypothesis of no relevant change is rejected at nominal level $5\%$.}
    \label{tab1}
\end{table}

\begin{figure}
    \centering
    \includegraphics[width=9cm,height=6cm]{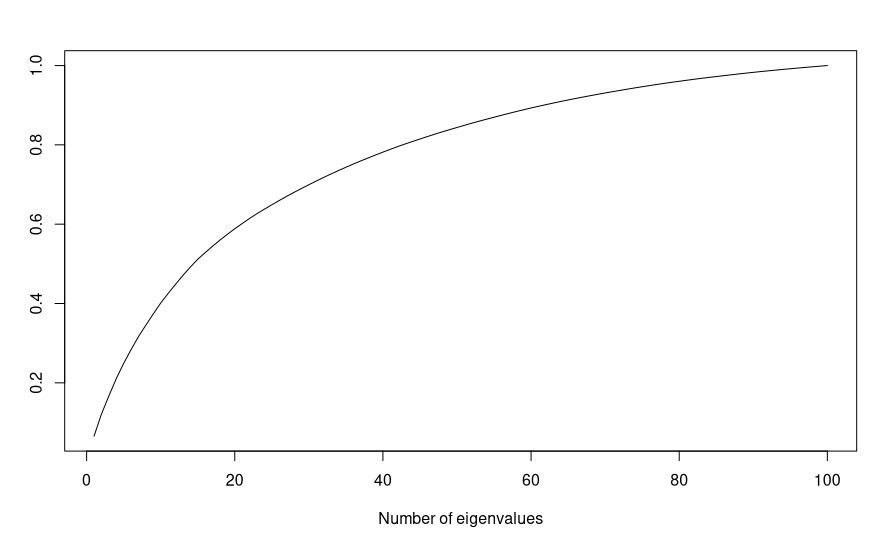}
    \caption{\it The function $e$ in \eqref{p1} for $j= 1,...,100$}
    \label{eigenvalues}
\end{figure}

Next we consider the test based on functional principal components developed in Section \ref{sec4}. 
In this case, the choice of 
$d$ is crucial and we display in Figure \ref{eigenvalues} the ratios
\begin{equation} \label{p1}
j  \mapsto e(j) :=
\frac{ \sum_{k=1}^j
\hat \lambda_k }{ {\rm trace }(\hat c_{\hat \vartheta_n} ) },
\end{equation}
where  $\hat \lambda_1, \hat \lambda_2 , \ldots $
are the eigenvalues of the estimated covariance operator \eqref{hd8a}. We observe that the eigenvalues are slowly decreasing, and we choose
$d=51$, which results in a value of $e(51) \approx 0.849 $ of explained variance.
The results of the test \eqref{disp:cp_mean_testrule} for the relevant hypotheses \eqref{hd9} are shown in Table \ref{tab2} for different values of $\Delta$ and $\alpha$. We observe that 
 the maximal threshold in \eqref{hd9}, such that the null hypothesis of no relevant change is rejected at nominal level $5\%$, is given by  
$4.467$.
Finally one and two-sided confidence intervals for the quantity $\big (\sum_{k=1}^d \innpr{\delta, w_k}^2\big )^{1/2} $ are obtained in same way as described in Remark \ref{remarkh0}
and are given by  $\sqbr{0, 3.254}$ and $\sqbr{1.877, 3.396}$, respectively  ($d=51$).

\begin{table}[!ht]
    \centering
    \begin{tabular}{c | c | c | c}
         $\Delta$ & 90 \% & 95 \% & 99 \%  \\
         \hline
         4.467 & reject & reject & accept\\
         4.468 & reject & accept & accept
    \end{tabular}
    \caption{\it 
        Results of the test \eqref{disp:cp_mean_testrule} with $d=51$ principal components for the Canadian weather data for different nominal level and different thresholds. The value $4.467$ represents  the maximal threshold in \eqref{hd9}, such that the null hypothesis of no relevant change is rejected at nominal level $5\%$.}
    \label{tab2}
\end{table}

\bigskip
\textbf{Acknowledgements}  
This research has been supported by the German Research Foundation (DFG), project number 45723897.

\bibliographystyle{apalike}
    \bibliography{lit}

\begin{thebibliography}{}

\bibitem[Aston and Kirch, 2012]{astonkirch2012}
Aston, J.~A. and Kirch, C. (2012).
\newblock Detecting and estimating changes in dependent functional data.
\newblock {\em Journal of Multivariate Analysis}, 109:204--220.

\bibitem[Aue et~al., 2022]{aue2019twosample}
Aue, A., Dette, H., and Rice, G. (2022).
\newblock Two-sample tests for relevant differences in the eigenfunctions of
  covariance operators. {\it \uppercase{s}tatistica \uppercase{s}inica}. to
  appear, ar\uppercase{X}iv:1909.06098.

\bibitem[Aue et~al., 2018]{Aue2015DatingSB}
Aue, A., Rice, G., and S\"onmez, O. (2018).
\newblock Detecting and dating structural breaks in functional data without
  dimension reduction.
\newblock {\em Journal of the Royal Statistical Society: Series B (Statistical
  Methodology)}, 80(3):509--529.

\bibitem[Benko et~al., 2009]{benkohaerdle2009}
Benko, M., Härdle, W., and Kneip, A. (2009).
\newblock {Common functional principal components}.
\newblock {\em The Annals of Statistics}, 37(1):1--34.

\bibitem[Berkes et~al., 2009]{berkes2009}
Berkes, I., Gabrys, R., Horv\'ath, L., and Kokoszka, P. (2009).
\newblock Detecting changes in the mean of functional observations.
\newblock {\em Journal of the Royal Statistical Society: Series B (Statistical
  Methodology)}, 71(5):927--946.

\bibitem[Berkes et~al., 2013]{berkes}
Berkes, I., Horv{\'a}th, L., and Rice, G. (2013).
\newblock Weak invariance principles for sums of dependent random functions.
\newblock {\em Stochastic Processes and their Applications}, 123(2):385--403.

\bibitem[Chow and Liu, 1992]{chowliu1992}
Chow, S.-C. and Liu, P.-J. (1992).
\newblock {\em Design and Analysis of Bioavailability and Bioequivalence
  Studies}.
\newblock Marcel Dekker, New York.

\bibitem[Dette et~al., 2020a]{dette2018}
Dette, H., Kokot, K., and Aue, A. (2020a).
\newblock Functional data analysis in the \uppercase{B}anach space of
  continuous functions.
\newblock {\em Annals of Statistics}, 48(2):1168--1192.

\bibitem[Dette et~al., 2020b]{dette}
Dette, H., Kokot, K., and Volgushev, S. (2020b).
\newblock Testing relevant hypotheses in functional time series via
  self-normalization.
\newblock {\em Journal of the Royal Statistical Society (B)}, 82(3):629--660.

\bibitem[Dette et~al., 2018]{detmolvolbre2015}
Dette, H., M{\"o}llenhoff, K., Volgushev, S., and Bretz, F. (2018).
\newblock Equivalence of regression curves.
\newblock {\em Journal of the American Statistical Association}, 113:711--729.

\bibitem[Dette and Wied, 2014]{Dette2014b}
Dette, H. and Wied, D. (2014).
\newblock Detecting relevant changes in time series models.
\newblock {\em Journal of the Royal Statistical Society}, 78(2):371--394.

\bibitem[Diamond et~al., 2013]{Diam13}
Diamond, H.~J., Karl, T.~R., Palecki, M.~A., Baker, C.~B., Bell, J.~E., Leeper,
  R.~D., Easterling, D.~R., Lawrimore, J.~H., Meyers, T.~P., Helfert, M.~R.,
  Goodge, G., and Thorne, P.~W. (2013).
\newblock U.\uppercase{S}. \uppercase{C}limate \uppercase{R}eference
  \uppercase{N}etwork after \uppercase{O}ne \uppercase{D}ecade of
  \uppercase{O}perations: \uppercase{S}tatus and \uppercase{A}ssessment.
\newblock {\em Bulletin of the American Meteorological Society}, 94(4):485 --
  498.

\bibitem[Fogarty and Small, 2014]{fogarty2014}
Fogarty, C.~B. and Small, D.~S. (2014).
\newblock Equivalence testing for functional data with an application to
  comparing pulmonary function devices.
\newblock {\em Ann. Appl. Stat.}, 8(4):2002--2026.

\bibitem[Gromenko et~al., 2017]{gromenko}
Gromenko, O., Kokoszka, P., and Reimherr, M. (2017).
\newblock Detection of change in the spatiotemporal mean function.
\newblock {\em Journal of the Royal Statistical Society (B)}, 79(1):29--50.

\bibitem[Gsteiger et~al., 2011]{gsteiger2011}
Gsteiger, S., Bretz, F., and Liu, W. (2011).
\newblock Simultaneous confidence bands for nonlinear regression models with
  application to population pharmacokinetic analyses.
\newblock {\em Journal of Biopharmaceutical Statistics}, 21(4):708--725.

\bibitem[Hariz et~al., 2007]{hariz2007}
Hariz, S.~B., Wylie, J.~J., and Zhang, Q. (2007).
\newblock Optimal rate of convergence for nonparametric change-point estimators
  for nonstationary sequences.
\newblock {\em The Annals of Statistics}, 35:1802--1826.

\bibitem[Horv\'ath and Kokoszka, 2012]{HorvathKokoskza2012}
Horv\'ath, L. and Kokoszka, P. (2012).
\newblock {\em Inference for {F}unctional {D}ata with {A}pplications}.
\newblock Springer-Verlag, New York.

\bibitem[Jandhyala et~al., 2013]{jandhyala2013}
Jandhyala, V., Fotopoulos, S., MacNeill, I., and Liu, P. (2013).
\newblock Inference for single and multiple change-points in time series.
\newblock {\em Journal of Time Series Analysis}, 34(4):423--446.

\bibitem[Liu et~al., 2009]{liubrehaywynn2009}
Liu, W., Bretz, F., Hayter, A.~J., and Wynn, H.~P. (2009).
\newblock Assessing non-superiority, non-inferiority of equivalence when
  comparing two regression models over a restricted covariate region.
\newblock {\em Biometrics}, 65(4):1279--1287.

\bibitem[Wellek, 2010]{wellek2010testing}
Wellek, S. (2010).
\newblock {\em Testing {Statistical} {Hypotheses} of {Equivalence} and
  {Noninferiority}}.
\newblock Chapman and Hall/CRC.

\bibitem[Zhang and Shao, 2015]{zhangshao2015}
Zhang, X. and Shao, X. (2015).
\newblock Two sample inference for the second-order property of temporally
  dependent functional data.
\newblock {\em Bernoulli}, 21(2):909--929.

\bibitem[Zhang et~al., 2011]{zhang2011}
Zhang, X., Shao, X., Hayhoe, K., and Wuebbles, D.~J. (2011).
\newblock {Testing the structural stability of temporally dependent functional
  observations and application to climate projections}.
\newblock {\em Electronic Journal of Statistics}, 5:1765--1796.

\bibitem[Zhao et~al., 2021]{zhao2021composite}
Zhao, Z., Ma, T.~F., Ng, W.~L., and Yau, C.~Y. (2021).
\newblock A composite likelihood-based approach for change-point detection in
  spatio-temporal process. ar\uppercase{X}iv:1904.06340.

\end{thebibliography}

\appendix
\section{Appendix: proofs} 
\label{sec6} 
  \def\theequation{A.\arabic{equation}}
\setcounter{equation}{0}
\subsection{Some preliminary results}

We start with some preparations and present several results which are used in the proof. We define for $s \in S$ and $t \in \sqbr{0,1}$ $$S_n (s, t, \lambda) := \frac{1}{n} \sum_{i=1}^{\gbr{\lambda n}} \br{X_i (s,t) - \mu (s,t)}$$
and state the following  result,
which can be  obtained by  generalizing Theorem 1 in \cite{berkes} to the space
$L^2 (S \times [0,1])$.
\begin{thm}\label{generalberkes}
{\it 
If 
 Assumption \ref{ass_2nd} is satisfied,
 there exists  a sequence  of Gaussian processes $\big ( \cbr{\Gamma_n (s, t, \lambda) \mid s \in S, 0 \leq \lambda, t \leq 1} \big )_{n \in \mathbb{N}}$, such that
    \begin{align*}
        \sup_{0 \leq \lambda \leq 1} \sum_{s \in S} \int \br{ \sqrt{n} S_n (s, t, \lambda) - \Gamma_n (s, t, \lambda)}^2 \mathrm{d}t = \sup_{0 \leq \lambda \leq 1} \norm{\sqrt{n} S_n (\cdot, \cdot, \lambda) - \Gamma_n (\cdot, \cdot, \lambda)}^2 = o_\p (1)
    \end{align*}
    and 
    $$\cbr{\Gamma_n (s, t, \lambda) \mid s \in S, 0 \leq \lambda, t \leq 1} \eqo{$\mathpcal{D}$} \cbr{\Gamma (s, t, \lambda) \mid s \in S, 0 \leq \lambda, t \leq 1},
    $$
    where
    $$\Gamma (s, t, \lambda) := \sum_{i=1}^\infty \sqrt{\lambda_i} \phi_i (s, t) W_i (\lambda),$$ 
    $\lambda_i, \phi_i$ are the eigenvalues and eigenvectors of the covariance operator of $X_j$ and $\{ W_i\}_{i \in \mathbb{N}}$ are independent  standard Brownian motions.
    }
\end{thm}

Our next auxiliary result quantifies the difference  between the processes 
$\{D_n (\lambda, \vartheta_0)\}_{\lambda \in [0,1]} $ and
$\{D_n (\lambda, \hat{\vartheta}_n)\}_{\lambda \in [0,1]} $ as $n \to \infty$.

\begin{lem}\label{lem:approx_for_cp_in_summation}
{\it 
    If Assumption \ref{ass_2nd} is satisfied, then
    \begin{enumerate}
        \item[(1)] $\displaystyle \sup_{0 \leq \lambda \leq 1} \norm{D_n (\lambda, \vartheta)} = O_\p (1)$ for $\vartheta \in \{ \vartheta_0, \hat{\vartheta}_n \}$ ,
        \item[(2)]  $\displaystyle \sup_{0 \leq \lambda \leq 1} \norm{D_n (\lambda, \hat{\vartheta}_n) - D_n (\lambda, \vartheta_0) } = o_\p (n^{-1/2})$.
    \end{enumerate}
    }
\end{lem}

\begin{proof}
    We can assume that $\frac{1}{n} \leq \lambda$, because, by definition, $D_n (\lambda, \vartheta_0) \equiv D_n (\lambda, \hat{\vartheta}_n ) \equiv0$  
    if $0 \leq \lambda < \frac{1}{n}$,
    and both assertions are trivially true.

For a proof of part (1) we note  that for some $\kappa > 0$
        \begin{align}\label{disp:aue_lemma_b1}
            \sup_{\frac{1}{n} \leq \lambda \leq 1} \frac{1}{\sqrt{\gbr{\lambda n}}} \norm{\sum_{j=1}^{\gbr{\lambda n}} \eta_j^{(\ell)}} = O_\p (\log^{1/\kappa} (n))~~~(\ell =1,2),
        \end{align}
       which follows from Lemma B.1 in  the online supplementary material of \cite{Aue2015DatingSB}.
        Let $\lambda \in \sqbr{1/n, 1}$. It suffices to show the assertion for $\vartheta = \vartheta_0$, because the second part of the Lemma implies the statement for $\vartheta = \hat \vartheta_n$. Then 
        $$\norm{D_n (\lambda, \vartheta_0)} \leq \norm{\tilde{D}_n (\lambda, \vartheta_0)} + \lambda \norm{\delta} + o_\p (1) = \norm{\tilde{D}_n (\lambda, \vartheta_0)} + O_\p (1),
        $$
        where the process  $\{\tilde{D}_n (\lambda, \vartheta) \}_{\lambda \in [0,1]}$  in $L^2 (S \times \sqbr{0,1})$ is defined by 
       \begin{align*}
    \tilde{D}_n (\lambda, \vartheta) & := \frac{1}{\gbr{n \vartheta}} \sum_{j=1}^{\gbr{\lambda \gbr{n \vartheta}}} \eta^{(1)}_j  - \frac{1}{n - \gbr{n \vartheta}} \sum_{j=\gbr{n \vartheta} + 1}^{\gbr{n \vartheta} + \gbr{\lambda (n - \gbr{n \vartheta})}} \eta^{(2)}_j .
\end{align*}
Note that $\tilde  D_n(\lambda, \vartheta)$ is the centered version of  ${D}_n (\lambda, \vartheta)$  defined in \eqref{hd3} and   that  
\begin{equation} \label{hd20}
    D_n (\lambda, \vartheta_0) + \lambda \delta = \tilde{D}_n (\lambda, \vartheta_0) + O_\p (n^{-1})
    \end{equation}
    (note that  $\vartheta_0 \in (\varepsilon, 1-\varepsilon)$ and that 
    $\varepsilon \geq \frac{1}{n}$ if $n$ is sufficiently large).
        We have
        \begin{align*}
            \sup_{\frac{1}{n} \leq \lambda \leq 1} \norm{\tilde{D}_n (\lambda, \vartheta_0)} \leq \sup_{\frac{1}{n} \leq \lambda \leq 1} \frac{1}{\gbr{n \vartheta_0}} \norm{\sum_{j=1}^{\gbr{\lambda \gbr{n \vartheta_0}}} \eta_j^{(1)} } + \sup_{\frac{1}{n} \leq \lambda \leq 1} \frac{1}{n - \gbr{n \vartheta_0}} \norm{\sum_{j=\gbr{n \vartheta_0} +1 }^{\gbr{n \vartheta_0} + \gbr{\lambda (n - \gbr{n \vartheta_0})}} \eta_j^{(2)} },
        \end{align*}
        where  the first term 
     can be estimated as follows
        \begin{align*}
            \sup_{\frac{1}{n} \leq \lambda \leq 1} \frac{1}{\gbr{n \vartheta_0}} \norm{\sum_{j=1}^{\gbr{\lambda \gbr{n \vartheta_0}}} \eta_j^{(1)} } &\leq \frac{1}{\sqrt{\gbr{n \vartheta_0}}} \sup_{\frac{1}{n} \leq \lambda \leq 1} \frac{1}{\sqrt{\gbr{\lambda \gbr{n \vartheta_0}}}} \norm{\sum_{j=1}^{\gbr{\lambda \gbr{n \vartheta_0}}} \eta_j^{(1)} }\\
            &\leq \frac{1}{\sqrt{\gbr{n \vartheta_0}}} \sup_{\frac{1}{n} \leq \lambda \leq 1} \frac{1}{\sqrt{\gbr{\lambda n}}} \norm{\sum_{j=1}^{\gbr{\lambda n}} \eta_j^{(1)} } = O_\p \br{\frac{\log^{1/\kappa} (n)}{\sqrt{n}}} = o_\p (1).
        \end{align*}
       (here in the second inequality, we expanded the set over which we take the supremum and in the last step, we used the estimate \eqref{disp:aue_lemma_b1}).
     A  similar argument provides the same rate for the second term, which proves the part (1) of  Lemma \ref{lem:approx_for_cp_in_summation}.
        
        
   For the proof of the second assertion  we need the following (slightly more general) statement from the beginning of Section  B.1 in \cite{dette}, which states that  the sequence of processes $\cbr{\Gamma_n (s, t, \lambda) \mid s \in S, 0 \leq t, \lambda \leq 1}_{n \in \mathbb{N}}$
   in Theorem \ref{generalberkes} satisfies 
        \begin{align}\label{disp:sup_gamma_cont}
            \sup_{\substack{\nu, \lambda \in \sqbr{0,1}:\\ |\nu - \lambda| \leq \kappa_n}} \norm{\Gamma_n (\cdot, \cdot, \nu) - \Gamma_n (\cdot, \cdot, \lambda)}^2 = o_\p (1).
        \end{align}
        for any and  positive sequence $\br{\kappa_n}_{n \in \mathbb{N}}$ with $\kappa_n \to 0$.
        By adding and subtracting $\lambda \delta$, we have $D_n (\lambda, \hat{\vartheta}_n) - D_n (\lambda, \vartheta_0) = \tilde{D}_n (\lambda, \hat{\vartheta}_n) - \tilde{D}_n (\lambda, \vartheta_0) + O_\p (n^{-1})$, where 
        \begin{align*}
            \tilde{D}_n (\lambda, \hat{\vartheta}_n) - \tilde{D}_n (\lambda, \vartheta_0) &= \frac{1}{\gbr{n \hat{\vartheta}_n}} \sum_{j=1}^{\gbr{\lambda \gbr{n \hat{\vartheta}_n}}} \eta_j^{(1)} - \frac{1}{\gbr{n \vartheta_0}} \sum_{j=1}^{\gbr{\lambda \gbr{n \vartheta_0}}} \eta_j^{(1)}\\
            &\pppad+ \frac{1}{n - \gbr{n \vartheta_0}} \sum_{j=\gbr{n \vartheta_0} + 1}^{\gbr{n \vartheta_0} + \gbr{\lambda (n - \gbr{n \vartheta_0}) }} \eta_j^{(2)} - \frac{1}{n - \gbr{n \hat{\vartheta}_n}} \sum_{j=\gbr{n \hat{\vartheta}_n} + 1}^{\gbr{n \hat{\vartheta}_n} + \gbr{\lambda (n - \gbr{n \hat{\vartheta}_n}) }} \eta_j^{(2)}~.
        \end{align*}
    For the   first difference
    we obtain by a similar argument as in the proof of part (1) that 
        \begin{align*}
            &\sup_{\frac{1}{n} \leq \lambda \leq 1} \sqrt{n} \norm{\frac{1}{\gbr{n \hat{\vartheta}_n}} \sum_{j=1}^{\gbr{\lambda \gbr{n \hat{\vartheta}_n}}} \eta_j^{(1)} - \frac{1}{\gbr{n \vartheta_0}} \sum_{j=1}^{\gbr{\lambda \gbr{n \vartheta_0}}} \eta_j^{(1)} }\\
            &\pppad \pppad \pppad \leq \frac{n}{\gbr{n \hat{\vartheta}_n}} \sup_{\frac{1}{n} \leq \lambda \leq 1} \norm{\frac{1}{\sqrt{n}}  \sum_{j=1}^{\gbr{\lambda \gbr{n \hat{\vartheta}_n}}} \eta_j^{(1)} - \Gamma_n \br{\cdot, \cdot, \lambda \tfrac{\gbr{n \hat{\vartheta}_n}}{n}}} \\
            &\pppad\pppad\pppad 
            + \frac{n}{\gbr{n \vartheta_0}} \sup_{\frac{1}{n} \leq \lambda \leq 1} \norm{\frac{1}{\sqrt{n}}  \sum_{j=1}^{\gbr{\lambda \gbr{n \vartheta_0}}} \eta_j^{(1)} - \Gamma_n \br{\cdot, \cdot, \lambda \tfrac{\gbr{n \vartheta_0}}{n}}} \\
            &\pppad\pppad \pppad + \frac{n}{\gbr{n \hat{\vartheta}_n}} \sup_{\frac{1}{n} \leq \lambda \leq 1}  \norm{\Gamma_n \br{\cdot, \cdot, \lambda \tfrac{\gbr{n \hat{\vartheta}_n}}{n} } - \Gamma_n (\cdot, \cdot, \lambda \hat{\vartheta}_n)}
            \\
            &\pppad \pppad \pppad + \frac{n}{\gbr{n\vartheta_0}} \sup_{\frac{1}{n} \leq \lambda \leq 1} \norm{\Gamma_n \br{\cdot, \cdot, \lambda \tfrac{\gbr{n \vartheta_0}}{n}} - \Gamma_n (\cdot, \cdot, \lambda \vartheta_0)}\\
            &\pppad\pppad \pppad+ \sup_{\frac{1}{n} \leq \lambda \leq 1} \norm{ \frac{n}{\gbr{n \hat{\vartheta}_n}} \Gamma_n (\cdot, \cdot, \lambda \hat{\vartheta}_n) - \frac{n}{\gbr{n \vartheta_0}} \Gamma_n (\cdot, \cdot, \lambda \vartheta_0)}\\
            &\pppad \pppad\pppad \leq \sup_{\substack{\nu, \lambda \in \sqbr{0,1}:\\ |\nu - \lambda| \leq \frac{1}{n}}} \norm{\Gamma_n (\cdot, \cdot, \nu) - \Gamma_n (\cdot, \cdot, \lambda)}  \\
            &\pppad\pppad \pppad + O(1) \negthickspace \negthickspace \sup_{\substack{\nu, \lambda \in \sqbr{0,1}:\\|\nu - \lambda| \leq |\hat{\vartheta}_n - \vartheta_0|}} \norm{\Gamma_n (\cdot, \cdot, \nu) - \Gamma_n (\cdot, \cdot, \lambda) } + \frac{|\vartheta_0 - \hat{\vartheta}_n|}{\vartheta_0 \hat{\vartheta}_n} O_\p (1) +  o_\p (1)\\
            &\pppad \pppad\pppad = o_\p (1),
        \end{align*}
        where we have used (\ref{disp:sup_gamma_cont}) in the last equality.
       Assertion (2) of Lemma \ref{lem:approx_for_cp_in_summation} now follows by a similar argument for the second difference.
\end{proof}

We conclude our preparations recalling the definition of the inner product in \eqref{det1}
  and state  a lemma regarding the weak convergence of the process
     \begin{align}
       \label{hd11a}  
        \br{{Z}_n^{(1)}, {Z}_n^{(2)}}^\top & :=
       \big \{ (  {Z}_n^{(1)} (\lambda),   
        {Z}_n^{(2)} (\lambda) )^\top \big \}_{\lambda \in [0,1]}  \\
        & \nonumber 
        := \Big \{  \Big ( \frac{1}{\sqrt{n}} \sum_{i=1}^{\gbr{n \lambda}} \br{\innpr{\eta_i^{(k)} \otimes \eta_i^{(k)} - c^{(k)}{}, \xi_k}_2 + \innpr{\eta_i^{(k)}, \zeta_k}} \Big )_{k=1,2} \Big \}_{\lambda \in [0,1]}  ~, 
    \end{align}
    where 
  $\{\eta^{(1)}_j  \}_{j\in \mathbb{Z}}$ and  $\{\eta^{(2)}_j \}_{j \in \mathbb{Z}}$ are centered processes  in  
  $ L^2 (S \times \sqbr{0,1})$ 
  and 
  $\zeta_1, \zeta_2 \in L^2 (S \times \sqbr{0,1})$, $\xi_1, \xi_2 \in L^2 ((S \times \sqbr{0,1})^2)$ are given functions.
 We emphasize  that we consider the process $    ( {Z}_n^{(1)}, {Z}_n^{(2)})^\top $ with 
different   parameters $\zeta_1, \zeta_2 \in L^2 (S \times \sqbr{0,1})$, $\xi_1, \xi_2 \in L^2 ((S \times \sqbr{0,1})^2)$ in the proofs of the results of Section \ref{sec3} and \ref{sec4}.
 The proof of the following result is similar  to  the proof of Lemma B.1 in \cite{dette} and  therefore omitted.

\begin{lem}\label{general_lemma_b1}
{\it 
    Let $\zeta_1, \zeta_2 \in L^2 (S \times \sqbr{0,1})$, $\xi_1, \xi_2 \in L^2 ((S \times \sqbr{0,1})^2)$ be fixed but arbitrary functions and let 
    $\{\eta^{(1)}_j  \}_{j\in \mathbb{Z}},$ $\{\eta^{(2)}_j \}_{j \in \mathbb{Z}}$
    denote  centered processes  satisfying Assumption \ref{ass_4th}.
    Then the process defined in \eqref{hd11a}
    converges weakly in $\ell^{\infty} (\sqbr{0,1})^2$, that is
    \begin{align} \label{hd11b}
        \br{{Z}_n^{(1)}, {Z}_n^{(2)}}^\top \convproc \Sigma^{1/2} \br{\mathbb{B}_1, \mathbb{B}_2}^\top ,
    \end{align}
    where 
    $\mathbb{B}_1, \mathbb{B}_2$ are independent Brownian motions and $\Sigma  =(\Sigma_{kl})_{k,l=1,2}$ is a $2 \times 2$ matrix with entries
    \begin{align} \nonumber
        \Sigma_{kl} = \sum_{i\in \mathbb{Z}} &\bigg(\cov (\innpr{\eta_0^{(k)} \otimes \eta_0^{(k)}, \xi_k}_2, \innpr{\eta_i^{(l)} \otimes \eta_i^{(l)}, \xi_l}_2) + \cov (\innpr{\eta_0^{(k)} \otimes \eta_0^{(k)}, \xi_k}_2, \innpr{\eta_i^{(l)}, \zeta_l})\\
        & \pad  +\cov (\innpr{\eta_0^{(k)}, \zeta_k}, \innpr{\eta_i^{(l)} \otimes \eta_i^{(l)}, \xi_l}_2) + \cov (\innpr{\eta_0^{(k)}, \zeta_k}, \innpr{\eta_i^{(l)}, \zeta_l}) \bigg).
        \label{hd19}
    \end{align}
Moreover, in the  case  $\xi_1 \equiv \xi_2 \equiv 0$  Assumption \ref{ass_2nd} instead of Assumption \ref{ass_4th} is sufficient for  the weak convergence in \eqref{hd11b}.
}
\end{lem}

\subsection{Proof of Theorem \ref{thm:cp_spatiotemporal_mean}}
 
For a proof of the weak convergence in \eqref{hd6a} we use \eqref{hd20} and  several applications of the Cauchy-Schwarz inequality, which give 
            \begin{align*}
                \mathbb{U}_n (\lambda) &:= \sqrt{n} \br{\norm{D_n (\lambda, \vartheta_0)}^2 - \lambda^2 \norm{\delta}^2} = \tfrac{1}{\sqrt{n}} \norm{\sqrt{n} \tilde{D}_n (\lambda, \vartheta_0)}^2 + 2\lambda \innpr{\sqrt{n} \tilde{D}_n (\lambda, \vartheta_0), \delta} + o_\p (1)\\
                &= 2\lambda \innpr{\sqrt{n} \tilde{D}_n (\lambda, \vartheta_0), \delta} + o_\p (1),
            \end{align*}
            since $\norm{\sqrt{n} \tilde{D}_n (\lambda, \vartheta_0)}^2 = O_\p (1)$ by Theorem \ref{generalberkes}. Next define for $\ell=1,2$ and $\lambda \in \sqbr{0,1}$
            \begin{align*}
                {Z}^{(\ell)}_n (\lambda) := \frac{1}{\sqrt{n}} \sum_{j=1}^{\gbr{\lambda n}} \innpr{\eta_j^{(\ell)}, \delta},
            \end{align*}
            in order to rewrite
            \begin{align*}
                \innpr{\sqrt{n} \tilde{D}_n (\lambda, \vartheta_0), \delta} &= \frac{n}{\gbr{n \vartheta_0}} \frac{1}{\sqrt{n}} \sum_{i=1}^{\gbr{\lambda \gbr{n \vartheta_0}}} \innpr{ \eta_i^{(1)} , \delta} - \frac{n}{n - \gbr{n \vartheta_0}} \frac{1}{\sqrt{n}} \sum_{i=\gbr{n \vartheta_0} + 1}^{\gbr{n \vartheta_0} + \lambda (n - \gbr{n \vartheta_0})} \innpr{ \eta_i^{(2)} , \delta}\\
                &= \frac{n}{\gbr{n \vartheta_0}} {Z}^{(1)}_n \br{\lambda \tfrac{\gbr{n \vartheta_0}}{n}} - \frac{n}{n - \gbr{n \vartheta_0}} \br{{Z}^{(2)}_n \br{\tfrac{\gbr{n \vartheta_0} + \gbr{\lambda (n - \gbr{n \vartheta_0})}}{n}} - {Z}^{(2)}_n \br{\vartheta_0}}.
            \end{align*}
            Hence,
            \begin{align*}
              \mathbb{U}_n (\lambda)  &=   \sqrt{n} \br{\norm{D_n (\lambda, \vartheta_0)}^2 - \lambda^2 \norm{\delta}^2}  \\
              & = \frac{2 \lambda}{\vartheta_0} {Z}^{(1)}_n \br{\lambda \tfrac{\gbr{n \vartheta_0}}{n}} - \frac{2 \lambda}{1-\vartheta_0} \br{{Z}^{(2)}_n \br{\tfrac{\gbr{n \vartheta_0} + \gbr{\lambda (n - \gbr{n \vartheta_0})}}{n}} - {Z}^{(2)}_n \br{\vartheta_0}} + o_\p (1)
            \end{align*}
            and by an application of Lemma \ref{general_lemma_b1} with $\zeta_1 = \zeta_2 = \delta$ and $\xi_1 \equiv \xi_2 \equiv 0$, (note that  in this case  Assumption \ref{ass_2nd} is sufficient) we obtain
            \begin{align*}
                \cbr{
                    \mathbb{U}_n (\lambda) 
            }_{\lambda \in \sqbr{0,1}} \convproc & \cbr{\mathbb{Z}(\lambda)}_{\lambda \in \sqbr{0,1}} \\
            & := \cbr{\frac{2 \lambda}{\vartheta_0} {\mathbb{Z}}^{(1)} (\lambda \vartheta_0) - \frac{2\lambda}{1-\vartheta_0} \br{{\mathbb{Z}}^{(2)} (\vartheta_0 + \lambda (1- \vartheta_0)) - {\mathbb{Z}}^{(2)} (\vartheta_0)}}_{\lambda \in \sqbr{0,1}},
            \end{align*}
            with ${\mathbb{Z}}^{(1)} = {\Sigma}_{11}^* \mathbb{B}_1 + {{\Sigma}}_{12}^* \mathbb{B}_2$ and ${\mathbb{Z}}^{(2)} = {\Sigma}_{21}^* \mathbb{B}_1 + {\Sigma}_{22}^* \mathbb{B}_2,$
            where $\mathbb{B}_1$  and 
             $\mathbb{B}_2$ are  independent Brownian motions. Here  ${\Sigma}_{ij}^*$ denotes the $ij$-th entry of the matrix  $\Sigma^{1/2}$ ($i,j=1,2$) and $\Sigma$ is defined by the entries given in \eqref{hd19}, where in this case  $\zeta_1 = \zeta_2 = \delta$ and $\xi_1 \equiv \xi_2 \equiv 0$.
            Inspecting the covariance structure of the limit above, we see that because of
            \begin{align*}
                \cov ({\mathbb{Z}}^{(1)} (\lambda_1 \vartheta_0), {\mathbb{Z}}^{(1)} (\lambda_2 \vartheta_0)) &= (\lambda_1 \wedge \lambda_2) \vartheta_0 \Sigma_{11},\\
                \cov ({\mathbb{Z}}^{(2)} (\vartheta_0 + \lambda_1 (1 - \vartheta_0)) - {\mathbb{Z}}^{(2)} ( \vartheta_0), {\mathbb{Z}}^{(2)} (\vartheta_0 + \lambda_2 (1 - \vartheta_0)) - {\mathbb{Z}}^{(2)} (\vartheta_0)) &= (\lambda_1 \wedge \lambda_2) (1-\vartheta_0) \Sigma_{22},\\
                \cov ({\mathbb{Z}}^{(1)} (\lambda_1 \vartheta_0), {\mathbb{Z}}^{(2)} (\vartheta_0 + \lambda_2 (1 - \vartheta_0)) - {\mathbb{Z}}^{(2)} ( \vartheta_0)) &= 0,
            \end{align*}
            we have $$\mathrm{Cov} \br{\mathbb{Z} (\lambda_2) , \mathbb{Z}(\lambda_2)} = 4 \lambda_1 \lambda_2 (\lambda_1 \wedge \lambda_2) \br{\frac{1}{\vartheta_0} \Sigma_{11} + \frac{1}{1 - \vartheta_0}\Sigma_{22}}.$$ Hence the limit has the same distribution as $\tau_{\delta, \vartheta_0} \cbr{\lambda \mathbb{B} (\lambda)}_{\lambda \in \sqbr{0,1}}$, and the first assertion follows.\\
          
            \noindent For a proof of \eqref{hd6b}, we rewrite the expression as
            \begin{align*}
                \sqrt{n}  \br{\norm{D_n (\lambda, \hat{\vartheta}_n)}^2 - \lambda^2 \norm{\delta}^2} &= \sqrt{n} \big( \norm{D_n (\lambda, {\vartheta_0})}^2 - \lambda^2 \norm{\delta}^2 \big) + N_n (\lambda),
            \end{align*}
            where $$N_n (\lambda) := \sqrt{n} \bigg( \norm{D_n (\lambda, \hat{\vartheta}_n) - D_n (\lambda, {\vartheta_0})}^2 + 2 \innpr{D_n (\lambda, \hat{\vartheta}_n), D_n (\lambda, \hat{\vartheta}_n) - D_n (\lambda, {\vartheta_0})} \bigg)$$
            By the Cauchy-Schwarz inequality and both parts of Lemma \ref{lem:approx_for_cp_in_summation}, we obtain $\sup_{\lambda} N_n (\lambda) = o_\p (1)$. Therefore, the convergence follows from \eqref{hd6a}.

    \bigskip
    \bigskip

 \subsection{Proof of Theorem \ref{thm:test_asymptotics}}
 
 First of all, we consider $\| \delta \|^2 = 0$. By a careful 
 inspection of the proof part 1 of Lemma \ref{lem:approx_for_cp_in_summation}, we can verify that $\hat{\mathbb{D}}_n$ is a consistent estimator of $\| \delta \|^2$. So in this case, we have $\hat{\mathbb{D}}_n = o_\p (1)$. By Remark \ref{remsym} the random variable 
 $\mathbb{W}$ in \eqref{disp:random_variable_w}
 has a symmetric distribution, which in turn implies $q_{1-\alpha} (\mathbb{W}) \geq 0$, whenever $\alpha \leq 0.5$.
 Together with the fact that $\hat{\mathbb{V}}_n \geq 0$, we have
 \begin{align*}
    \P{\hat{\mathbb{D}}_n > \Delta + q_{1-\alpha} (\mathbb{W}) \hat{\mathbb{V}}_n} \longrightarrow 0,
 \end{align*}
 whenever $\Delta > 0$  and $\alpha \leq 0.5$. Now, consider the case $\| \delta \| > 0$, then, by  Theorem \ref{thm:cp_spatiotemporal_mean}, the weak convergence of  the vector 
 $\sqrt{n}  \big (( \hat{\mathbb{D}}_n - \norm{\delta}^2 ) ,  \hat{\mathbb{V}}_n \big )^\top $ is an immediate consequence of \eqref{hd6b} and the continuous mapping theorem applied to the map
 $$f \longmapsto 
 \big (f(1),  \big ( \int_0^1 \lambda^2 (f(\lambda) - \lambda f(1)) \mathrm{d} \nu (\lambda)  \big )^{1/2} \big)^\top.
 $$
 Hence, we have $ \hat{\mathbb{D}}_n - \norm{\delta}^2 = o_\p (1)$ and $\hat{\mathbb{V}}_n = o_\p (1)$. Therefore,
 \begin{align*}
     \P{\hat{\mathbb{D}}_n - \norm{\delta}^2 > \Delta - \norm{\delta}^2 + q_{1-\alpha} (\mathbb{W}) \hat{\mathbb{V}}_n}
     \longrightarrow \begin{cases} 0,& \text{if } \norm{\delta}^2 < \Delta,\\
     1,& \text{if } \norm{\delta}^2 > \Delta,\end{cases}
 \end{align*}
 as $n \to \infty$.
 In the case $\norm{\delta}^2 = \Delta$ and $\tau_{\delta, \vartheta_0}^2 > 0$, 
 we obtain from  \eqref{hd6b} and the continuous mapping theorem applied to the map $f \longmapsto 
 f(1) \big ( \int_0^1 \lambda^2 (f(\lambda) - \lambda f(1)) \mathrm{d} \nu (\lambda)  \big )^{-1/2}
 $, that
 \begin{align*}
     \frac{\hat{\mathbb{D}}_n - \norm{\delta}^2}{\hat{\mathbb{V}}_n} \convproc \mathbb{W},
 \end{align*}
 where $\mathbb{W}$ is defined in \eqref{disp:random_variable_w}. This yields
 \begin{align*}
  \lim_{n \to \infty}    \P{\hat{\mathbb{D}}_n > \Delta  + q_{1-\alpha} (\mathbb{W}) \hat{\mathbb{V}}_n} =
    \lim_{n \to \infty}  
    \P{\frac{\hat{\mathbb{D}}_n - \norm{\delta}^2}{\hat{\mathbb{V}}_n} > q_{1-\alpha} (\mathbb{W})} = \alpha.
 \end{align*}

 \subsection{Proof of Theorem \ref{thm7}}
    We show that
    \begin{align}
        \sup_{0 \leq \lambda \leq 1} \sqrt{\lambda} \norm{\hat{\vartheta}_n \hat{c}^{(1)}_\lambda - \vartheta_0 \E {\eta_0^{(1)} \otimes \eta_0^{(1)}}}_2 & = O_\p \br{\frac{\log^{2/\kappa} (n)}{\sqrt{n}}} ~, \label{hd13a} \\
                \sup_{0 \leq \lambda \leq 1} \sqrt{\lambda} \norm{
              (1- \hat{\vartheta}_n) \hat{c}_\lambda^{(2)} - (1-\vartheta_0 )\E {\eta_0^{(2)} \otimes \eta_0^{(2)}}}_2 & = O_\p \br{\frac{\log^{2/\kappa} (n)}{\sqrt{n}}} ~, \label{hd13b}
    \end{align}
    For the sake of brevity we restrict ourselves to a proof of \eqref{hd13a}, the proof of \eqref{hd13b} follows by similar arguments.
  
    First, we assume that $\lambda < \frac{1}{\gbr{n \hat \vartheta_n}}$, then $\hat{c}^{(1)}_{\lambda} = 0$ and 
    \begin{align*}
        \sqrt{\lambda} \norm{\hat{\vartheta}_n \hat{c}^{(1)}_\lambda - \vartheta_0 \E {\eta_0^{(1)} \otimes \eta_0^{(1)}}}_2 < \frac{\vartheta_0}{\sqrt{\gbr{n \hat \vartheta_n}}} \bigg| \E {\eta_0^{(1)} \otimes \eta_0^{(1)}} \bigg| = O_\p (n^{-1/2})
    \end{align*}
    uniformly in $\lambda \in [0, \frac{1}{\gbr{n \hat \vartheta_n}})$. If  $\frac{1}{\gbr{n \hat \vartheta_n}} \leq \lambda \leq 1$ we use the representation 
    \begin{align}\label{disp:expansion}
        \hat{c}^{(1)}_\lambda = \frac{1}{\gbr{\lambda \gbr{n \hat{\vartheta}_n}}} \sum_{i=1}^{\gbr{\lambda \gbr{n \hat{\vartheta}_n}}} X_i \otimes X_i - \frac{1}{\gbr{\lambda \gbr{n \hat{\vartheta}_n}}^2 } \sum_{i,j=1}^{\gbr{\lambda \gbr{n \hat{\vartheta}_n}}} X_i \otimes X_j
    \end{align}
   and consider the first term.
   For this we get
    \begin{align} \nonumber 
      &  \frac{1}{\gbr{\lambda \gbr{n \hat{\vartheta}_n}}} \sum_{i=1}^{\gbr{\lambda \gbr{n \hat{\vartheta}_n}}} X_i \otimes X_i = \frac{1}{\gbr{\lambda \gbr{n \hat{\vartheta}_n}}} \sum_{i=1}^{\gbr{\lambda \gbr{n \hat{\vartheta}_n}} \wedge \gbr{n \vartheta_0}} X_i \otimes X_i 
     \\
     &  
    \nonumber 
    ~~~~~~     ~~~~~~  ~~~~~~    ~~~~~~  ~~~~~~    ~~~~~~  ~~~~~~  ~~~~~~  ~~~~~~    + \frac{1}{\gbr{\lambda \gbr{n \hat{\vartheta}_n}}} \sum_{i=\gbr{\lambda \gbr{n \hat{\vartheta}_n}} \wedge \gbr{n \vartheta_0} + 1}^{\gbr{\lambda \gbr{n \hat{\vartheta}_n}}} X_i \otimes X_i\\
        &    \nonumber 
        ~~~~~~  ~~~~~~  = \frac{1}{\gbr{\lambda \gbr{n \hat{\vartheta}_n}}} \sum_{i=1}^{\gbr{\lambda \gbr{n \hat{\vartheta}_n}} \wedge \gbr{n \vartheta_0}} \br{\eta_i^{(1)} \otimes \eta_i^{(1)} + \mu \otimes \eta_i^{(1)} + \eta_i^{(1)} \otimes \mu + \mu \otimes \mu} \\
        &  \label{hd14b} 
        ~~~~~~   ~~~~~~ 
        + \frac{1}{\gbr{\lambda \gbr{n \hat{\vartheta}_n}}} \sum_{i=\gbr{\lambda \gbr{n \hat{\vartheta}_n}} \wedge \gbr{n \vartheta_0} + 1}^{\gbr{\lambda \gbr{n \hat{\vartheta}_n}}} \br{\eta_i^{(2)} \otimes \eta_i^{(2)} + \tilde{\mu} \otimes \eta_i^{(2)} + \eta_i^{(2)} \otimes \tilde{\mu} + \tilde{\mu} \otimes \tilde{\mu}},
    \end{align}
    where  we use the notation $\tilde{\mu} := \mu + \delta$.  The second term in (\ref{disp:expansion}) can be rewritten as follows: 
    \begin{align*}
        \frac{1}{\gbr{\lambda \gbr{n \hat{\vartheta}_n}}^2 }& \sum_{i,j=1}^{\gbr{\lambda \gbr{n \hat{\vartheta}_n}}} X_i \otimes X_j\\
        &= \frac{1}{\gbr{\lambda \gbr{n \hat{\vartheta}_n}}^2 } \sum_{i,j=1}^{\gbr{\lambda \gbr{n \hat{\vartheta}_n}} \wedge \gbr{n \vartheta_0}} \br{\eta_i^{(1)} \otimes \eta_j^{(1)} + \mu \otimes \eta_j^{(1)} + \eta_i^{(1)} \otimes \mu + \mu \otimes \mu}\\
        & + \frac{\mathbbm{1} \cbr{\gbr{n \vartheta_0} \leq \gbr{\lambda \gbr{n \hat{\vartheta}_n}}}}{\gbr{\lambda \gbr{n \hat{\vartheta}_n}}^2} \bigg[ \sum_{i=1}^{\gbr{n \vartheta_0}} \sum_{j= \gbr{n \vartheta_0}+1}^{\gbr{\lambda \gbr{n \hat{\vartheta}_n}}} \br{\eta_i^{(1)} \otimes \eta_j^{(2)} + \mu \otimes \eta_j^{(2)} + \eta_i^{(1)} \otimes \tilde{\mu} + \mu \otimes \tilde{\mu}}\\
        &\phantom{asdfasdfasdasdfsdfa} + \sum_{i= \gbr{n \vartheta_0}+1}^{\gbr{\lambda \gbr{n \hat{\vartheta}_n}}} \sum_{j=1}^{\gbr{n \vartheta_0}}  \br{\eta_i^{(2)} \otimes \eta_j^{(1)} +  \eta_i^{(2)} \otimes \mu +  \tilde{\mu} \otimes \eta_j^{(1)} + \tilde{\mu} \otimes \mu}\\
        &\phantom{asdfasdfaasdsdfsdfa} + \sum_{i,j= \gbr{n \vartheta_0}+1}^{\gbr{\lambda \gbr{n \hat{\vartheta}_n}}} \br{\eta_i^{(2)} \otimes \eta_j^{(2)} +  \eta_i^{(2)} \otimes \tilde{\mu} +  \tilde{\mu} \otimes \eta_j^{(2)} + \tilde{\mu} \otimes \tilde{\mu}} \bigg].
    \end{align*}
    We now consider the cases
    $\gbr{\lambda \gbr{n \hat{\vartheta}_n}} < \gbr{n \vartheta_0}$
    and $\gbr{\lambda \gbr{n \hat{\vartheta}_n}} \geq \gbr{n \vartheta_0}$ separately. 
    For this purpose, we define the set $\Lambda := \{\frac{1}{\gbr{n \hat \vartheta_n}}\leq \lambda \leq 1 \mid \gbr{\lambda \gbr{n \hat \vartheta_n}} < \gbr{n \vartheta_0} \}$ and note that
    \begin{align*}
        &\sup_{\frac{1}{n} \leq \lambda \leq 1} \sqrt{\lambda} \norm{\hat{\vartheta}_n \hat{c}_\lambda^{(1)}  - \vartheta_0 \E{\eta_0^{(1)} \otimes \eta_0^{(1)}}}_2\\
    & \pppad =  \max \cbr{\sup_{\lambda \in \Lambda} \sqrt{\lambda} \norm{\hat{\vartheta}_n \hat{c}_\lambda^{(1)}  - \vartheta_0 \E{\eta_0^{(1)} \otimes \eta_0^{(1)}}}_2, \sup_{\lambda \in \Lambda^C} \sqrt{\lambda} \norm{\hat{\vartheta}_n \hat{c}_\lambda^{(1)}  - \vartheta_0 \E{\eta_0^{(1)} \otimes \eta_0^{(1)}}}_2}~.
    \end{align*}
  We consider each case individually, 
starting  with the first term, where the  
    $\sup$ is taken over the set $\Lambda$,  
i.e. $\gbr{\lambda \gbr{n \hat{\vartheta}_n}} < \gbr{n \vartheta_0}$. Observing that the sum in \eqref{hd14b} vanishes, we obtain
in this case that 
    \begin{align}
    \label{hd15}
        \hat{c}_\lambda^{(1)} = \frac{1}{\gbr{\lambda \gbr{n \hat{\vartheta}_n}}} \sum_{i=1}^{\gbr{\lambda \gbr{n \hat{\vartheta}_n}}} \eta_i^{(1)} \otimes \eta_i^{(1)} - \frac{1}{ \gbr{\lambda \gbr{n \hat{\vartheta}_n}}^2} \sum_{i,j=1}^{\gbr{\lambda \gbr{n \hat{\vartheta}_n}}} \eta_i^{(1)} \otimes \eta_j^{(1)}
     \end{align}
    In the subsequent  discussion, we will repeatedly make use of  the following inequality, without explicitly mentioning it. It  is  obtained by expanding the set over which  the supremum  is calculated, and substituting $\lambda' = \lambda \frac{\gbr{n \hat \vartheta_n}}{n}$:
     \begin{align*}
        \sup_{\lambda \in \Lambda} \frac{1}{\sqrt{\gbr{\lambda \gbr{n \hat \vartheta_n}}}} \norm{\sum_{i=1}^{\gbr{\lambda \gbr{n \hat \vartheta_n}}} \eta_i^{(1)}} &\leq \sup_{\frac{1}{n} \leq \lambda' \leq \frac{\gbr{n \hat \vartheta_n}}{n}} \frac{1}{\sqrt{\gbr{\lambda' n}}} \norm{\sum_{i=1}^{\gbr{\lambda' n}} \eta_i^{(1)}} 
        \leq \sup_{\frac{1}{n} \leq \lambda \leq 1} \frac{1}{\sqrt{\gbr{\lambda n }}} \norm{\sum_{i=1}^{\gbr{\lambda n}} \eta_i^{(1)}}.
    \end{align*}
Regarding the second term in \eqref{hd15} we have for some $\kappa > 0$
    \begin{align*}
        \sup_{\lambda \in \Lambda} \sqrt{\lambda} \hat{\vartheta}_n \norm{\frac{1}{ \gbr{\lambda \gbr{n \hat{\vartheta}_n}}^2} \sum_{i,j=1}^{\gbr{\lambda \gbr{n \hat{\vartheta}_n}}} \eta_i^{(1)} \otimes \eta_j^{(1)}}_2 &= \sup_{\lambda \in \Lambda} \frac{\lambda \hat{\vartheta}_n}{\gbr{\lambda \gbr{ n \hat{\vartheta}_n}}} \frac{1}{\sqrt{\lambda}} \norm{\frac{1}{ \sqrt{\gbr{\lambda \gbr{n \hat{\vartheta}_n}}}} \sum_{i=1}^{\gbr{\lambda \gbr{n \hat{\vartheta}_n}}} \eta_i^{(1)}}^2\\
        &\leq O_\p (n^{-1/2}) \br{\sup_{\frac{1}{n} \leq \lambda \leq 1} \frac{1}{ \sqrt{\gbr{\lambda n}}} \norm{ \sum_{i=1}^{\gbr{\lambda n}} \eta_i^{(1)}} }^2 \\
        & = O_\p \br{\frac{\log^{2/\kappa} (n)}{\sqrt{n}}},
    \end{align*}
     by (\ref{disp:aue_lemma_b1}).
    This  implies
    \begin{align*}
    &    \sup_{\lambda \in \Lambda} \sqrt{\lambda} \norm{\hat{\vartheta}_n \hat{c}_\lambda^{(1)}  - \vartheta_0 \E{\eta_0^{(1)} \otimes \eta_0^{(1)}}}_2 \\
    & ~~~~~~~~ 
    \leq \sup_{ \lambda \in \Lambda } \sqrt{\lambda} \norm{ \frac{\hat{\vartheta}_n}{\gbr{\lambda \gbr{n \hat{\vartheta}_n}}} \sum_{i=1}^{\gbr{\lambda \gbr{n \hat{\vartheta}_n}}} \br{\eta_i^{(1)} \otimes \eta_i^{(1)} - \E{\eta_0^{(1)} \otimes \eta_0^{(1)}}}}_2\\
        &\pppad\pppad ~~~~~~~~ ~~~~~~~~  + |\hat{\vartheta}_n - \vartheta_0 | \E{\eta_0^{(1)} \otimes \eta_0^{(1)}} + O_\p \br{\frac{\log^{2/\kappa} (n)}{\sqrt{n}}}\\
        &~~~~~~~~  \leq O_\p (n^{-1/2}) \sup_{\frac{1}{n} \leq \lambda \leq 1} 
        \frac{1}{\sqrt{\gbr{\lambda n}}} \norm{ \sum_{i=1}^{\gbr{\lambda n}} \br{\eta_i^{(1)} \otimes \eta_i^{(1)} - \E{\eta_0^{(1)} \otimes \eta_0^{(1)}}} }_2
        + O_\p \br{\frac{\log^{2/\kappa} (n)}{\sqrt{n}}}\\
        &~~~~~~~~  \leq O_\p \br{\frac{\log^{2/\kappa} (n)}{\sqrt{n}}},
    \end{align*}
    where we used \eqref{prop:changepoint_conv_rate} and again a variation of (\ref{disp:aue_lemma_b1}), this time for $L^2 ((S \times \sqbr{0,1})^2)$-valued random variables.

\medskip

\noindent
It remains to calculate the supremum over the set $\Lambda^C$. 
   If $\gbr{\lambda \gbr{n \hat{\vartheta}_n}} \geq \gbr{n \vartheta_0}$, a tedious
   but straightforward calculation gives
    \begin{align*}
        \hat{c}_\lambda^{(1)} &= \frac{1}{\gbr{\lambda \gbr{n \hat{\vartheta}_n}}} \sum_{i=1}^{\gbr{n \vartheta_0}} \eta_i^{(1)} \otimes \eta_i^{(1)} - \frac{1}{\gbr{\lambda \gbr{n \hat{\vartheta}_n}}^2} \sum_{i,j=1}^{\gbr{n \vartheta_0}} \eta_i^{(1)} \otimes \eta_j^{(1)} 
        \\
        & \ppad 
        + \frac{1}{\gbr{\lambda \gbr{n \hat{\vartheta}_n}}} \sum_{i= \gbr{n \vartheta_0} + 1}^{\gbr{\lambda \gbr{n \hat{\vartheta}_n}}} (\eta_i^{(2)} + \delta) \otimes (\eta_i^{(2)} + \delta) \\
        &\ppad - \frac{1}{\gbr{\lambda \gbr{n \hat{\vartheta}_n}}^2} \sum_{i=1}^{\gbr{n \vartheta_0}} \sum_{j=\gbr{n \vartheta_0} +1}^{\gbr{\lambda \gbr{n \hat{\vartheta}_n}}} \eta_i^{(1)} \otimes (\eta_j^{(2)} + \delta) - \frac{1}{\gbr{\lambda \gbr{n \hat{\vartheta}_n}}^2} \sum_{i=\gbr{n \vartheta_0} +1}^{\gbr{\lambda \gbr{n \hat{\vartheta}_n}}} \sum_{j=1}^{\gbr{n \vartheta_0}}  (\eta_i^{(2)} + \delta) \otimes \eta_j^{(1)}\\
        &\ppad + \frac{1}{\gbr{\lambda \gbr{n \hat{\vartheta}_n}}^2} \sum_{i,j= \gbr{n \vartheta_0}+1}^{\gbr{\lambda \gbr{n \hat{\vartheta}_n}}} (\eta_i^{(2)} + \delta) \otimes (\eta_j^{(2)} + \delta)
    \end{align*}
and  we claim that
   \begin{align} \nonumber
        \sup_{\lambda \in \Lambda^C} \sqrt{\lambda} \norm{\hat{\vartheta}_n \hat{c}_\lambda^{(1)} - \vartheta_0 \E{\eta_0^{(1)} \otimes \eta_0^{(1)}} }_2 &\leq \sup_{\frac{1}{n} \leq \lambda \leq 1} \sqrt{\lambda} \norm{\frac{\hat{\vartheta}_n}{\gbr{\lambda \gbr{n \hat{\vartheta}_n}}} \sum_{i=1}^{\gbr{n \vartheta_0}} \eta_i^{(1)} \otimes \eta_i^{(1)} - \vartheta_0 \E{\eta_0^{(1)} \otimes \eta_0^{(1)}} }_2\\
        &\pppad + O_\p \br{\frac{\log^{2/\kappa} (n)}{\sqrt{n}}}
        ~.
        \label{hd16}
    \end{align}
  For a proof of \eqref{hd16}, we show that
    \begin{align}
        \label{hd12a} 
      &  \sup_{ \lambda \in \Lambda^C} \norm{\frac{1}{\gbr{\lambda \gbr{n \hat{\vartheta}_n}} } \sum_{i=1}^{\gbr{n \vartheta_0}} \eta_i^{(1)} } = O_\p (n^{-1/2}),
        \\
        \label{hd12b} 
    &    \sup_{ \lambda \in \Lambda^C } \norm{\frac{1}{\gbr{\lambda \gbr{n \hat{\vartheta}_n}}} \sum_{i=\gbr{n \vartheta_0} + 1}^{\gbr{\lambda \gbr{n \hat{\vartheta}_n}}} (\eta_i^{(2)} + \delta) } = O_\p \br{\frac{\log^{1/\kappa} (n)}{\sqrt{n}}},
        \\
          \label{hd12c} 
    &      \sup_{\lambda \in \Lambda^C } \sqrt{\lambda} \hat{\vartheta}_n \norm{\frac{1}{\gbr{\lambda \gbr{n \hat{\vartheta}_n}}} \sum_{i= \gbr{n \vartheta_0} +1}^{\gbr{\lambda \gbr{n \hat{\vartheta}_n}}} (\eta_i^{(2)} + \delta) }^2 = O_\p \br{\frac{\log^{2/\kappa} (n)}{\sqrt{n}}}
        \\  
        \label{hd12d} 
      & \sup_{\lambda \in \Lambda^C } \sqrt{\lambda} \hat{\vartheta}_n \norm{\frac{1}{\gbr{\lambda \gbr{n \hat{\vartheta}_n}}} \sum_{i=\gbr{n \vartheta_0} + 1}^{\gbr{\lambda \gbr{n \hat{\vartheta}_n}}}  (\eta_i^{(2)} + \delta) \otimes (\eta_i^{(2)} + \delta)} = O_\p \br{\frac{\log^{1/\kappa} (n)}{\sqrt{n}}}
    \end{align}
    The estimate \eqref{hd12a} follows from
    \begin{align*}
        \sup_{\lambda \in \Lambda^C} \norm{\frac{1}{\gbr{\lambda \gbr{n \hat{\vartheta}_n}} } \sum_{i=1}^{\gbr{n \vartheta_0}} \eta_i^{(1)} } \leq \frac{1}{\sqrt{\gbr{n \vartheta_0}}} \norm{\frac{1}{\sqrt{\gbr{n \vartheta_0}}} \sum_{i=1}^{\gbr{n \vartheta_0}} \eta_i^{(1)} } = O_\p (n^{-1/2}),
    \end{align*}
    because $\gbr{n \vartheta_0} \leq \gbr{\lambda \gbr{n \hat{\vartheta}_n}}$ and the central limit theorem in Hilbert spaces in the last step. For the estimate \eqref{hd12b}, recall that by \eqref{prop:changepoint_conv_rate}, we have $\hat{\vartheta}_n = \vartheta_0 + o_\p (n^{-1/2})$.
    Moreover, this implies 
    \begin{align} \label{est1}
   \sup_{\lambda \in \Lambda^C} \frac{\gbr{n \vartheta_0}}{\gbr{\lambda \gbr{n \hat{\vartheta}_n}}} = 1 + o_\p (n^{-1/2}),
    \end{align}
    because 
    $$
    1 + o_\p (n^{-1/2}) = \frac{\gbr{n \vartheta_0}}{\gbr{n \hat{\vartheta}_n}} \leq \sup_{\lambda \in \Lambda^C} \frac{\gbr{n \vartheta_0}}{\gbr{\lambda \gbr{n \hat{\vartheta}_n}}} \leq \sup_{\lambda \in \Lambda^C} \frac{\gbr{\lambda \gbr{n \hat{\vartheta}_n}}}{\gbr{\lambda \gbr{n \hat{\vartheta}_n}}} = 1 ~.
    $$
    Using again $\gbr{n \vartheta_0} \leq \gbr{\lambda \gbr{n \hat{\vartheta}_n}}$ and the estimate  (\ref{disp:aue_lemma_b1}), we find that
    \begin{align*}
      &  \sup_{\lambda \in \Lambda^C} \norm{\frac{1}{\gbr{\lambda \gbr{n \hat{\vartheta}_n}}} \sum_{i=\gbr{n \vartheta_0} + 1}^{\gbr{\lambda \gbr{n \hat{\vartheta}_n}}} (\eta_i^{(2)} + \delta) } \\
      & ~~~~~~~~~~~~~~~ ~ 
      \leq \frac{1}{\sqrt{\gbr{n \vartheta_0}}} \sup_{\lambda \in \Lambda^C}  \norm{\frac{1}{\sqrt{\gbr{\lambda \gbr{n \hat{\vartheta}_n}}}} \sum_{i=1}^{\gbr{\lambda \gbr{n \hat{\vartheta}_n}}} \eta_i^{(2)} }\\
        &~~~~~~~~ ~~~~~~~~ \pppad + \frac{1}{\sqrt{\gbr{n \vartheta_0}}} \norm{\frac{1}{\sqrt{\gbr{n \vartheta_0}}} \sum_{i=1}^{\gbr{n \vartheta_0}} \eta_i^{(2)} } + 
\sup_{\lambda \in \Lambda^C}
\frac{\gbr{\lambda \gbr{n \hat{\vartheta}_n}} - \gbr{n \vartheta_0}}{\gbr{\lambda \gbr{n \hat{\vartheta}_n}}} \norm{\delta}
\\
        &~~~~~~~~ ~~~~~~~~ \leq O_\p (n^{-1/2}) \sup_{\frac{1}{n} \leq \lambda \leq 1} \norm{\frac{1}{\gbr{\lambda n}} \sum_{i=1}^{\gbr{\lambda n}} \eta_i^{(2)} }+ \frac{\gbr{n \hat \vartheta_n} - \gbr{n \vartheta_0}}{\gbr{n \vartheta_0}} 
        \norm{\delta} + O_\p (n^{-1/2})\\
        &~~~~~~~~ ~~~~~~~~ 
        = O_\p \br{\frac{\log^{1/\kappa} (n)}{\sqrt{n}}}.
    \end{align*}
    These calculations also yield  the estimate \eqref{hd12c}. For \eqref{hd12d}, we have, using  a version of (\ref{disp:aue_lemma_b1}) for $L^2( (S \times \sqbr{0,1})^2 )$-valued random variables and the fact that $\gbr{n \vartheta_0} \leq  \gbr{\lambda \gbr{n \hat{\vartheta}_n}}$, that
    \begin{align*}
        &\sup_{\lambda \in \Lambda^C} \sqrt{\lambda} \hat{\vartheta}_n \norm{\frac{1}{\gbr{\lambda \gbr{n \hat{\vartheta}_n}}} \sum_{i=\gbr{n \vartheta_0} +1}^{\gbr{\lambda \gbr{n \hat{\vartheta}_n}}} (\eta_i^{(2)} + \delta ) \otimes (\eta_i^{(2)} + \delta ) }_2\\
        &\pppad\leq \frac{1}{\sqrt{\gbr{n \vartheta}}} \sup_{\lambda \in \Lambda^C}  \norm{\frac{1}{\sqrt{\gbr{\lambda \gbr{n \hat{\vartheta}_n}}}} \sum_{i=1}^{\gbr{\lambda \gbr{n \hat{\vartheta}_n}}}  \tilde{\eta}^\otimes_i  }_2 + \frac{\hat \vartheta_n}{\sqrt{\gbr{n \vartheta_0}}} \norm{\frac{1}{\sqrt{\gbr{n \vartheta_0}}} \sum_{i=1}^{\gbr{n \vartheta_0}} \tilde{\eta}_i^\otimes } \\
        &\pppad\pppad+ o_\p (n^{-1/2}) \E{(\eta_0^{(2)} + \delta) \otimes (\eta_0^{(2)} + \delta)}\\
        &\pppad= O_\p \br{\frac{\log^{1/\kappa} (n)}{\sqrt{n}}},
    \end{align*}
    where $\tilde{\eta}^\otimes_i := (\eta_i^{(2)} + \delta) \otimes (\eta_i^{(2)} + \delta) - \E{(\eta_0^{(2)} + \delta) \otimes (\eta_0^{(2)} + \delta)}$. 
Finally,  observing \eqref{hd16}, the assertion \eqref{hd13a} follows from 
    \begin{align*}
        &\sup_{\lambda \in \Lambda^C} \sqrt{\lambda} \norm{\frac{\hat{\vartheta}_n}{\gbr{\lambda \gbr{n \hat{\vartheta}_n}}} \sum_{i=1}^{\gbr{n \vartheta_0}} \eta_i^{(1)} \otimes \eta_i^{(1)} - \vartheta_0 \E{\eta_0^{(1)} \otimes \eta_0^{(1)}} }_2\\
        &~~~~~~~~ \pppad \leq \frac{1}{\sqrt{\gbr{n \vartheta_0}}} \norm{ \frac{1}{\sqrt{\gbr{n \vartheta_0}}} \sum_{i=1}^{\gbr{n \vartheta_0}} \br{\eta_i^{(1)} \otimes \eta_i^{(1)} - \E{\eta_0^{(1)} \otimes \eta_0^{(1)}}} }_2 \\
        & ~~~~~~~~ ~~~~~~~~   \pppad + \sup_{\lambda \in \Lambda^C} \bigg| \frac{\gbr{n \vartheta_0}}{\gbr{\lambda \gbr{n \hat{\vartheta}_n}}}\hat{\vartheta}_n - \vartheta_0 \bigg| \E{\eta_0^{(1)} \otimes \eta_0^{(1)}}\\
        &~~~~~~~~  \pppad = O_\p (n^{-1/2}) 
    \end{align*}
    by \eqref{disp:aue_lemma_b1} and \eqref{est1}.
  This completes the proof of Theorem \ref{thm7}.

\subsection{Proof of Theorem \ref{thm9}} 
We begin stating an expansion for the difference $\hat{c}_{j,\lambda} \hat{w}_{j,\lambda} - w_j$, which is proved by similar arguments as given in  the proof of Proposition 2.1 in \cite{aue2019twosample}. The details are therefore omitted.  

\begin{prop} \label{prop12}
    Suppose Assumption \ref{ass_4th} and \ref{eigvals_order} hold. If further $\vartheta_0 \in \br{\varepsilon, 1-\varepsilon}$, then for any $j \leq d$ and some $\kappa >0$, we have 
    \begin{align*}
        \sup_{0 \leq \lambda \leq 1} \norm{\lambda (\hat{c}_{j,\lambda} \hat{w}_{j,\lambda} - w_j) - \frac{1}{\sqrt{n}} \sum_{i \neq j} \frac{w_i}{\tau_j - \tau_i}  \innpr{\hat{Z}_{n, \lambda}, w_j \otimes w_i}_2  } &= O_\p \br{\frac{\log^{\kappa} (n)}{n}}
    \end{align*}
    and
    \begin{align*}
        \sup_{0\leq\lambda\leq 1} \sqrt{\lambda} \| \hat c_{j, \lambda} \hat w_{j,\lambda} - w_{j} \| &= O_\p \br{\frac{\log^{1/\kappa} (n)}{\sqrt{n}}},
    \end{align*}
    where $\hat{c}_{j, \lambda} = \mathrm{sign} \innpr{\hat{w}_{j, \lambda}, w_j}$ and $$\hat{Z}_{n, \lambda} := \frac{1}{\sqrt{n}} \br{ \sum_{i=1}^{\gbr{\lambda \gbr{n \vartheta_0}}} (\eta_i^{(1)} \otimes \eta_i^{(1)} - c^{(1)}) + \sum_{i=\gbr{n \vartheta_0}+1}^{\gbr{n \vartheta_0} + \gbr{\lambda (n - \gbr{n \vartheta_0})}} (\eta_i^{(2)} \otimes \eta_i^{(2)} - c^{(2)})  }.$$
\end{prop}

\noindent
For the proof of the first assertion \eqref{hd18a} of Theorem \ref{thm9}, we  add
and subtract $\lambda \delta$, use \eqref{hd20} and obtain
\begin{align} \label{d1}
    \innpr{D_n (\lambda, \vartheta_0), \hat{w}_{k,\lambda}}^2 - \lambda^2 \innpr{\delta, w_k}^2 
    &= I_{k,n} (\lambda) + J_{k,n} (\lambda) + K_{k,n} (\lambda) + O_\p (n^{-1})~,
\end{align}
where $\hat{c}_{k,\lambda}$ is defined in Proposition \ref{prop12} and
\begin{align*}
    I_{k,n} (\lambda) &:= \innpr{\tilde{D}_n (\lambda, \vartheta_0) , \hat{c}_{k,\lambda} \hat{w}_{k,\lambda} }^2, \\
    J_{k,n} (\lambda) & := 2\lambda \innpr{\tilde{D}_n (\lambda, \vartheta_0), \hat{c}_{k,\lambda} \hat{w}_{k,\lambda}} \innpr{\delta, \hat{c}_{k,\lambda} \hat{w}_{k, \lambda}},\\
    K_{k,n} (\lambda) &:= \lambda^2 \innpr{\delta, \hat{c}_{k,\lambda} \hat{w}_{k,\lambda}}^2 - \lambda^2 \innpr{\delta, w_k}^2.
\end{align*}
By Theorem \ref{generalberkes}, $\sup_{0 \leq \lambda \leq 1} \|\sqrt{n}\tilde{D}_n (\lambda, \vartheta_0)\|$ is bounded and the Cauchy-Schwarz inequality, yields 
\begin{align*}  
    \sup_{0\leq\lambda\leq 1}\sqrt{n} I_{k,n} (\lambda) \leq \frac{1}{\sqrt{n}} \sup_{0 \leq \lambda \leq 1} \|\sqrt{n}\tilde{D}_n (\lambda, \vartheta_0)\|^2 \norm{\hat{w}_{k,\lambda}}^2 = o_\p (1)
\end{align*}
for all $k=1,...,d$. Moving to the second term $J_{k,n}$, we see that
\begin{align*}
\nonumber
    \sqrt{n} J_{k,n} (\lambda) &= \sqrt{n} 2 \lambda ( \innpr{\tilde{D}_n (\lambda, \vartheta_0), \hat{c}_{k, \lambda} \hat{w}_{k, \lambda} - w_k } \innpr{\delta, \hat{c}_{k, \lambda} \hat{w}_{k, \lambda} - w_k } +\innpr{\tilde{D}_n (\lambda, \vartheta_0), w_k } \innpr{\delta, \hat{c}_{k, \lambda} \hat{w}_{k, \lambda} - w_k }  \\
    \nonumber
    &+ \innpr{\tilde{D}_n (\lambda, \vartheta_0), \hat{c}_{k, \lambda} \hat{w}_{k, \lambda} - w_k } \innpr{\delta, w_k }  + \innpr{\tilde{D}_n (\lambda, \vartheta_0), w_k } \innpr{\delta, w_k })\\
    &= 2\lambda \innpr{\sqrt{n} \tilde{D}_n (\lambda, \vartheta_0) , w_k} \innpr{\delta, w_k} + o_\p (1), 
\end{align*}
uniformly in $\lambda$, by the second part of Proposition \ref{prop12}. Lastly,
we have by the first part of Proposition \ref{prop12}
\begin{align}  \nonumber
    \sqrt{n} K_{k,n} (\lambda) &= \innpr{\delta, \lambda (\hat{c}_{k, \lambda} \hat{w}_{k, \lambda} - w_k) }^2 + 2 \lambda \innpr{\delta, \lambda (\hat{c}_{k, \lambda} \hat{w}_{k, \lambda} - w_k)} \innpr{\delta, w_k}\\
    &=2 \lambda \innpr{\delta, w_k} \innpr{\delta, \frac{1}{\sqrt{n}}  \sum_{i \neq k} \frac{w_i}{\tau_k - \tau_i} \innpr{\hat{Z}_{n, \lambda} , w_k \otimes w_i }_2 } + o_\p (1). \label{d4}
\end{align}
 Recalling the notations of $w$ and $f$ in \eqref{h1} and \eqref{h2}, 
 and combining \eqref{d1} - \eqref{d4} we can rewrite
            \begin{align*}
           \mathbb{T}_n  (\lambda ) & :=  \sqrt{n} \sum_{k=1}^d \br{\innpr{D_n (\lambda, \vartheta_0), \hat{w}_{k,\lambda}}^2 - \lambda^2 \innpr{\delta, w_k}^2} \\ 
             & ~~~~~~
             = \frac{2 \lambda}{\sqrt{n}} \sum_{i=1}^{\gbr{\lambda \gbr{n \vartheta_0}}} \innpr{\eta_i^{(1)} \otimes \eta_i^{(1)} - c^{(1)}, f}_2  + \frac{2 \lambda}{\sqrt{n}} \sum_{i=\gbr{n \vartheta_0}+1}^{\gbr{n \vartheta_0} + \gbr{\lambda (n - \gbr{n \vartheta_0})}} \innpr{\eta_i^{(2)} \otimes \eta_i^{(2)} - c^{(2)}, f}_2 \\
                &~~~~~~  - \frac{2 \lambda \sqrt{n}}{\gbr{n \vartheta_0}} \sum_{i=1}^{\gbr{\lambda \gbr{n \vartheta_0}}} \innpr{\eta_i^{(1)}, w}  - \frac{2 \lambda \sqrt{n}}{n - \gbr{n \vartheta_0}} \sum_{i=\gbr{n \vartheta_0} +1}^{\gbr{n \vartheta_0} + \gbr{\lambda (n - \gbr{n \vartheta_0})}} \innpr{\eta_i^{(2)}, w} + o_\p (1)\\
                & ~~~~~~ = 2 \lambda \br{\tilde{Z}_n^{(1)} \br{\frac{\lambda \gbr{n \vartheta_0}}{n}} + \tilde{Z}_n^{(2)} \br{\frac{\gbr{n \vartheta_0} + \gbr{\lambda (n - \gbr{n \vartheta_0})}}{n}} - \tilde{Z}_n^{(2)} \br{\vartheta_0} } + o_\p (1),
            \end{align*}
            where for $\ell =1,2$ 
            \begin{align*}
                \tilde{Z}_n^{(\ell)} (\lambda) = \frac{1}{\sqrt{n}} \sum_{i=1}^{\gbr{\lambda n}} \br{\innpr{\eta_i^{(\ell)} \otimes \eta_i^{(\ell)} - c^{(\ell)}, \xi_\ell}_2 + \innpr{\eta_i^{(\ell)}, \zeta_\ell}}  
            \end{align*}
            with $\xi_1 = \xi_2 = f$, $\zeta_1 = - \frac{1}{\vartheta_0} w$ and $\zeta_2 = - \frac{1}{1-\vartheta_0} w$. 
         Now by Lemma \ref{general_lemma_b1}, we obtain
            \begin{align*}
         \{  \mathbb{T}_n  (\lambda )
         \}_{\lambda \in [0,1]}
             \convproc &  \big \{ \tilde{\mathbb{Z}} \br{\lambda} \big  \}_{\lambda \in [0,1]} \\ 
            &  := 
             \big \{ 2 \lambda \tilde{\mathbb{Z}}^{(1)} \br{\lambda \vartheta_0} + 2 \lambda \br{\tilde{\mathbb{Z}}^{(2)} \br{\vartheta_0 + \lambda (1- \vartheta_0)} - \tilde{\mathbb{Z}}^{(2)} \br{\vartheta_0}} \big  \}_{\lambda \in [0,1]} ,
            \end{align*}
            where $\tilde{\mathbb{Z}}^{(1)} (\lambda) :=  \tilde{\Sigma}_{11}^* \mathbb{B}_1 (\lambda) + \tilde{\Sigma}_{12}^* \mathbb{B}_2(\lambda)$ and $\tilde{\mathbb{Z}}^{(2)} (\lambda) := \tilde{\Sigma}_{21}^* \mathbb{B}_1 + \tilde{\Sigma}_{22}^* \mathbb{B}_2$. $\tilde{\Sigma}_{ij}^*$ denotes the $ij$-th entry of the matrix $\tilde{\Sigma}^{1/2}$ and the entries of $\tilde{\Sigma} =(\tilde{\Sigma}_{ij})_{i,j=1,2}$ are defined by \eqref{hd19}
            with 
            $\xi_1 = \xi_2 = f$, $\zeta_1 = - \frac{1}{\vartheta_0} w$ and $\zeta_2 = -\frac{1}{1-\vartheta_0} w$. Inspecting the covariance structure of the limiting process, we see that
            \begin{align*}
                \cov &\br{\tilde{\mathbb{Z}} (\lambda_1 ),  \tilde{\mathbb{Z}} (\lambda_2) }=  4\lambda_1\lambda_2 (\lambda_1 \wedge \lambda_2)(\tilde{\Sigma}_{11} \vartheta_0 + \tilde{\Sigma}_{22} (1-\vartheta_0))
            \end{align*}
            and hence the the weak convergence in \eqref{hd18a} follows.

\medskip
\noindent
        To prove  the second assertion \eqref{hd18b} 
        we note that we have by part 1 and 2 of Lemma \ref{lem:approx_for_cp_in_summation}
        \begin{align*}
            \sup_{0 \leq \lambda \leq 1}& (\innpr{D_n (\lambda, \hat{\vartheta}_n), \hat{w}_{k,\lambda}}^2 - \innpr{D_n (\lambda, \vartheta_0), \hat{w}_{k,\lambda}}^2)\\
                &=\sup_{0 \leq \lambda \leq 1} \br{\innpr{{D}_n (\lambda, \hat{\vartheta}_n) - {D}_n (\lambda, \vartheta_0), \hat{w}_{k, \lambda}}^2 + 2 \innpr{{D}_n (\lambda, \hat{\vartheta}_n) - {D}_n (\lambda, \vartheta_0), \hat{w}_{k, \lambda}} \innpr{{D}_n (\lambda, \vartheta_0), \hat{w}_{k, \lambda}}}\\
                &\leq \sup_{0 \leq \lambda \leq 1} \norm{{D}_n (\lambda, \hat{\vartheta}_n) - {D}_n (\lambda, \vartheta_0)}^2  \norm{\hat{w}_{k, \lambda}}^2 \\
                &\pppad\pppad  + 2\sup_{0 \leq \lambda \leq 1} \norm{{D}_n (\lambda, \hat{\vartheta}_n) - {D}_n (\lambda, \vartheta_0)} \norm{{D}_n (\lambda, \vartheta_0)} \norm{\hat{w}_{k,\lambda}}^2\\
                &= o_\p (n^{-1/2})
        \end{align*}
        
        Consequently, assertion \eqref{hd18b} follows from \eqref{hd18a}.

\end{document}